\documentclass{compositio}

\usepackage{amsmath,amsfonts,yhmath,hyperref}
\usepackage{amsthm}
\usepackage{amssymb}
\usepackage[latin9]{inputenc}
\usepackage[nottoc,notlof,notlot]{tocbibind}
\usepackage{mathtools}
\usepackage{enumitem}
\usepackage{tikz-cd}
\usepackage{multirow}

\usepackage{pgf,tikz}

\usetikzlibrary{arrows}
\usetikzlibrary[patterns]

\definecolor{qqqqff}{rgb}{0.,0.,1.}
\definecolor{cqcqcq}{rgb}{0.7529411764705882,0.7529411764705882,0.7529411764705882}
\definecolor{ttqqqq}{rgb}{0.2,0.,0.}
\definecolor{qqqqff}{rgb}{0.,0.,1.}
\definecolor{xdxdff}{rgb}{0.49019607843137253,0.49019607843137253,1.}
\definecolor{zzttqq}{rgb}{0.6,0.2,0.}
\definecolor{cqcqcq}{rgb}{0.7529411764705882,0.7529411764705882,0.7529411764705882}

\definecolor{yqyqyq}{rgb}{0.5019607843137255,0.5019607843137255,0.5019607843137255}
\definecolor{uuuuuu}{rgb}{0.26666666666666666,0.26666666666666666,0.26666666666666666}
\definecolor{xdxdff}{rgb}{0.49019607843137253,0.49019607843137253,1.}
\definecolor{qqqqff}{rgb}{0.,0.,1.}

 \font\ncsc=cmcsc10
 \font\ntt=cmtt12

\newcommand{\ZZ}{\mathbb{Z}}
\newcommand{\RR}{\mathbb{R}}
\newcommand{\CC}{\mathbb{C}}

\newcommand{\gen}[1]{\langle #1 \rangle}

\newcommand{\modpar}{\mathcal{M}_0(\Delta,N_\RR)}

\newcommand{\ori}{\mathfrak{or}}

\newcommand{\bino}[2]{\begin{pmatrix}
#1 \\
#2 \\
\end{pmatrix}}

\newtheorem{theo}{Theorem}[section]
\newtheorem*{theom}{Theorem}
\newtheorem{prop}[theo]{Proposition}

\newtheorem{lem}[theo]{Lemma}

\theoremstyle{definition}
\newtheorem{defi}[theo]{Definition}
\newtheorem{prob}[theo]{Problem}

\theoremstyle{remark}
\newtheorem{remark}[theo]{Remark}
\newenvironment{rem}[1]{
    \begin{remark}#1}{
    \xqed{\blacklozenge}\end{remark}
}

\theoremstyle{remark}
\newtheorem{example}[theo]{Example}
\newenvironment{expl}[1]{
    \begin{example}#1}{
    \xqed{\lozenge}\end{example}
}

\newcommand{\xqed}[1]{
    \leavevmode\unskip\penalty9999 \hbox{}\nobreak\hfill
    \quad\hbox{\ensuremath{#1}}}

\keywords{Enumerative geometry, tropical refined invariants\\ \textit{Data Statement:} I do not have any data to point.}
 
\begin{document}
 
 
\title{Refined count for rational tropical curves \\ in arbitrary dimension}
\author{Thomas Blomme}

\begin{abstract}
In this paper we introduce a refined multiplicity for rational tropical curves in arbitrary dimension, which generalizes the refined multiplicity introduced by F.~Block and L.~G\"ottsche in \cite{block2016refined}. We then prove an invariance statement for the count of rational tropical curves in several enumerative problems using this new refined multiplicity. This leads to the definition of Block-G\"ottsche polynomials in any dimension.
\end{abstract}

\maketitle

\tableofcontents

\section{Introduction}

\subsection{Enumerative geometry, classical and tropical}

Tropical enumerative geometry takes its roots in the founding paper \cite{mikhalkin2005enumerative} by G.~Mikhalkin. He there proved a correspondence theorem that allows one to compute several classical geometric invariants for toric surfaces, namely Gromov-Witten type invariants and Welschinger invariants whenever defined, through the solving of some associated enumerative problem in planar tropical geometry. Other approaches to the correspondence theorem were later provided by E.~Shustin \cite{shustin2002patchworking} and I.~Tyomkin \cite{tyomkin2012tropical} along with a generalization for rational curves in higher dimensional toric varieties by T.~Nishinou and B.~Siebert \cite{nishinou2006toric}. Some progress has even been made in the non-toric case, for instance in the case of abelian surfaces by T.~Nishinou \cite{nicaise2018tropical}.

In each situation, the spirit of the tropical approach is the same: the goal is to make a count of algebraic curves satisfying some constraints and which is known to be invariant when the constraints vary in some family, \textit{e.g.} a count of curves passing through some points and which does not depend on the choice of the points as long as it is generic. One then degenerates the constraints to the so-called \textit{tropical limit}, where the classical enumerative problem becomes a tropical enumerative problem. The desired curve count might then be recovered by counting the tropical curves solution to this new problem with a suitable multiplicities depending on the problem.

\begin{expl}
Explicitly, in the case of rational curves in the projective plane $\CC P^2$, one of these enumerative problems consists in finding the degree $d$ rational curves passing through a chosen configuration of $3d-1$ points chosen in $\CC P^2$. The result does not depend on the chosen points if their choice is generic, and is known as $N_d$, a Gromov-Witten invariant of $\CC P^2$.

Over the real numbers, J-Y.~Welschinger showed in \cite{welschinger2005invariants} that if the configuration of $3d-1$ points is chosen invariant by conjugation and the curves are counted with a suitable sign, the count only depends on $d$ and the number $s$ of pairs of complex conjugated points. It is denoted by $W_{d,s}$. In particular $W_{d,0}$ corresponds to the signed count of real rational curves passing through a generic configuration of $3d-1$ real points.
\end{expl}

We keep in mind the case of degree $d$ rational curves in the pursue of this introduction. To the tropical limit, the enumerative problems leading to the definition of $N_d$ and $W_{d,0}$ degenerate to the following tropical enumerative problem: finding degree $d$ rational tropical curves passing through a chosen configuration of $3d-1$ points. For a definition of tropical curves, see section \ref{section tropical curves}, but for the reading of the introduction, one only needs to know that they are some finite graphs inside an affine space $\RR^r$, and that they have some infinite edges whose slopes form the \textit{degree} of the curve. According to \cite{mikhalkin2005enumerative}, to recover both invariants $N_d$ and $W_{d,0}$, it suffices to count tropical curves $\Gamma$ passing through a generic choice of $3d-1$ points with different multiplicities $m_\Gamma^\CC$ and  $m_\Gamma^\RR$, defined as follows: there exists a natural vertex multiplicity $m_V$, and one then sets
$$m_\Gamma^\CC=\prod_V m_V \text{ and } m_\Gamma^\RR=\prod_V \left\{ \begin{array}{l}
(-1)^{\frac{m_V-1}{2}} \text{ if }m_V\equiv 1\ [2] \\
0 \text{ else.}
\end{array}\right.$$
Both multiplicities are a product over the vertices of the tropical curve. The count with $m^\CC_\Gamma$ gives $N_d$ and the count with $m_\Gamma^\RR$ gives $W_{d,0}$. In \cite{block2016refined}, F.~Block and L.~G\"ottsche introduced a refined multiplicity for tropical curves, which is also a product over the vertices of the tropical curve:
$$m^{BG}_\Gamma=\prod_V \frac{q^{m_V/2}-q^{-m_V/2}}{q^{1/2}-q^{-1/2}}\in\ZZ[q^{\pm 1/2}].$$
This multiplicity is a Laurent polynomial in one variable $q$, which specializes to $m_\Gamma^\CC$ and $m_\Gamma^\RR$ when $q$ goes respectively to $1$ and $-1$. The count of tropical solutions using this refined multiplicity was proven to also lead to an invariant by I.~Itenberg and G.~Mikhalkin in \cite{itenberg2013block}. The classical meaning of these tropically defined invariants remains quite mysterious. Conjecturally \cite{gottsche2014refined}, they relate to the computation of the $\chi_{-y}$-genera of the relative Hilbert schemes of some linear systems of curves. Some attempts in this direction have been made by J.~Nicaise, S.~Payne and F.~Schroeter \cite{nicaise2018tropical}.

Another lead to understand these invariants was provided by Mikhalkin in \cite{mikhalkin2017quantum}. Though his approach is valid for curves in a general toric surface, we still consider the degree $d$ rational curves in the case of the projective plane, and a configuration $\mathcal{P}$ of $3d-1$ points which are this time chosen on the coordinate axis of the projective plane. Mikhalkin introduced a signed count for the oriented real rational curves passing through $\pm p$ for each $p\in\mathcal{P}$. This count is refined according to the value of a so called \textit{quantum index}, which is defined more generally for each oriented real curve of type I. He then proved that the result of the count only depends on the number of pairs of complex conjugated points on each axis. If the points are all chosen real, he relates the refined count to the tropical invariant obtained solving the analog tropical enumerative problem, and using Block-G\"ottsche multiplicities. The tropical analog problem is finding degree $d$ rational tropical curves whose infinite edges are contained in some fixed lines. The result was extended by the author in \cite{blomme2020tropical} and \cite{blomme2020computation} in case there are complex conjugated points.

Since, many other tropical enumerative problems have been shown to support a \textit{refinement}, \textit{i.e.} the existence of a non-trivial polynomial multiplicity such that the count of solution using this multiplicity does not depend on the chosen constraints. See for instance \cite{blechman2019refined}, \cite{schroeter2018refined} or  \cite{gottsche2019refined}. However, in every of these situations, the new refined multiplicity of a tropical curve is a product over its vertices, just as is the complex multiplicity. In higher dimension, the complex multiplicity is also not necessarily a product over the vertices. This a priori prevents anyone to find a simple multiplicity which is a product over the vertices and leads to an invariant. Finding a refined multiplicity which is not a product over the vertices seems even more complicated. Here, we show an in-between approach, providing a multiplicity which is up to sign a product over the vertices. The definition of this sign, which is at the center of the proof, carries the "non-vertex" information.

\begin{rem}
Notice that it might happen even in the planar setting that the complex multiplicity of a tropical curve is not a product over its vertices, as shown by T. Mandel and H. Ruddat \cite{mandel2019tropical}. This occurs when dealing with tropical curves which are not trivalent, for instance, imposing that the curves meets some line at one of its vertices.
\end{rem}

\subsection{Moment problem and refined invariance}

A natural generalization of the tropical enumerative problem considered by Mikhalkin in \cite{mikhalkin2017quantum} is the following, which we call $\omega$-\textit{problem}. Let $N$ be a lattice of rank $r$ and $M$ its dual lattice. From now on, $N_\RR=N\otimes_\ZZ\RR$ will be our workspace for tropical curves and every enumerative problem. We consider tropical curves of degree $\Delta$, which is a multiset of vectors in $N$ of total sum $0$, and not containing the zero vector. Let $\omega$ be a $2$-form on $N$, and for each $n_e\in\Delta$, fix a hyperplane $\mathcal{H}_e$ with slope $\ker\iota_{n_e}\omega$. We look for degree $\Delta$ rational tropical curves such that the unbounded end directed by $n_e$ is contained in $\mathcal{H}_e$. As this is not enough to ensure having a finite number of solutions, we add the condition that a chosen unbounded end $e_0$ is contained in a fixed plane $D\subset N_\RR$ transverse to $\mathcal{H}_{e_0}$.

\begin{rem}
In the planar case, there is only one possible choice of non-zero $2$-form $\omega$ up to multiplication by a scalar, and no need to choose a plane $D$. For each choice of $\Delta$, the $\omega$-problems coincide with the enumerative problem considered by Mikhalkin in \cite{mikhalkin2017quantum}.
\end{rem}

Such an enumerative problem where tropical curves have constraints imposed on their unbounded ends by a $2$-form also appear in the work of T. Mandel \cite{mandel2015scattering}. In his paper, the author deals with the refinement of some classical invariants appearing in the balancing of some scattering diagrams, thus expanding to higher dimension the refined approach of S.A. Filippini and G. Stoppa \cite{filippini2012block} to the tropical vertex group considered by M. Gross, B. Siebert and R. Pandharipande \cite{gross2010tropical}.

The tropical solutions to the $\omega$-problem are trivalent, and are counted with the following multiplicity:
$$B_\Gamma^{K_\omega}=\prod_{V}(q^{a_V\wedge b_V}-q^{b_V\wedge a_V})\in\ZZ[\Lambda^2 N/K_\omega ],$$
where $a_V$ and $b_V$ are two among the three outgoing slopes at the vertex $V$, chosen such that $\omega(a_V,b_V)>0$. The space $K_\omega$ is the space generated by the $a_V\wedge b_V$ in the kernel of $\omega:a\wedge b\in\Lambda^2 N\mapsto\omega(a,b)\in\ZZ$. Implicitly in the notation, the exponents of the variable $q$ are taken modulo $K_\omega$. For a generic choice of $\omega$, one has $K_\omega=0$ and there is no quotient. This multiplicity is a  product over the vertices of the curve, and can be seen as a Laurent polynomial in several variables, whose number depends on the rank of the quotient lattice $\Lambda^2 N/K_\omega$. This multiplicity also appears in \cite{shustin2018refined} to prove a local invariance statement in a similar but different setting. One recovers the refined multiplicity from Mandel \cite{mandel2015scattering} by evaluating the $2$-form $\omega$ on the exponents, and divising by a suitable power of $q-q^{-1}$. However, and surprisingly, such an evaluation is not necessary to get an invariance statement.

\begin{rem}
Getting back to the planar case, denoting by $\mathrm{det}$ the $2$-form, and identifying $\Lambda^2 N=\ZZ\det$ with $\ZZ$, the refined multiplicity $B_\Gamma^{K_\mathrm{det}}$ is given by
$$B_\Gamma^{K_\mathrm{det}}=\prod_V (q^{m_V}-q^{-m_V}).$$
It is thus equal to the refined multiplicity of Block-G\"ottsche up to a factor $q-q^{-1}$ at each vertex and a dilatation by $2$ in the exponents. It comes from the fact that $q^m-q^{-m}$ is always divisible by $q-q^{-1}$, but in the general case, there is no such canonical choice to divide all vertex multiplicities.
\end{rem}

We have the following invariance statement, asserting that the curve count does not depend on the choice of the hyperplanes $\mathcal{H}_e$ nor the choice of $D$.

\begin{theom}[\ref{theorem invariance moment problem}]
The count of solutions to the $\omega$-problem using the refined multiplicity $B_\Gamma^{K_\omega}$ does not depend on the choice of $\mathcal{H}_e$ and $D$ as long as it is generic. It only depends on the choice of $\omega$ and the unbounded end $e_0$.
\end{theom}

In particular the curve count does not depend on the choice of $D$, including the choice of its slope. The obtained invariant is denoted by $\mathcal{B}_\Delta^{\omega,e_0}$. Thus, we get a family of invariants $\mathcal{B}_\Delta^{\omega,e_0}\in\ZZ[\Lambda^2 N/K_\omega]$. Furthermore, we prove that this map is continuous in the following sense. The map $\omega\mapsto K_\omega$ is constant on the cones of some fan $\Omega_\Delta$ in $\Lambda^2 M_\RR$. Denoting $K_\sigma$ the value on some cone $\sigma$ of $\Omega_\Delta$, for $\tau\prec\sigma$, one has $K_\sigma\subset K_\tau$, leading to a projection $p_\sigma^\tau:\ZZ[\Lambda^2 N/K_\sigma]\rightarrow \ZZ[\Lambda^2 N/K_\tau]$. We then have the following theorem.

\begin{theom}[\ref{theorem continuity moment}]
The function $\omega\mapsto\mathcal{B}_\Delta^{\omega,e_0}$ is constant on the interior of the cones of $\Omega_\Delta$, and for $\tau\prec\sigma$, the value on $\tau$ is the image of the value on $\sigma$ under the projection $p_\sigma^\tau$.
\end{theom}

\subsection{Main results}

Theorem \ref{theorem invariance moment problem} is an invariance result with little improvement compared to already known tropical refined invariants since its proof relies on the fact that for the $\omega$-problem, the complex multiplicity of a tropical curve is still a product over its vertices, just as in the planar case. To some extent, it might even be seen as a particular case of the balancing of scaterring diagrams done by Mandel in \cite{mandel2015scattering}. However, this especially simple observation can then be used to provide an invariance statement in a larger class of enumerative problems for which the complex multiplicity is no longer a product over the vertices.

Consider rational tropical curves of degree $\Delta$ in $N_\RR$, and for each $n_e\in\Delta$, choose $L_e$ a sublattice of $N/\gen{n_e}$ such that $\sum_e \mathrm{cork}L_e=|\Delta|+r-3$. Let $\mathcal{L}_e$ be a affine subspace of $N_\RR/\gen{n_e}$ with slope $L_e$. We look for rational tropical curves of degree $\Delta$ such that image of the unbounded end $e$ in $N_\RR/\gen{n_e}$ lies in $\mathcal{L}_e$. The dimension count ensures that if the $\mathcal{L}_e$ are chosen generically, there is a finite number of solutions. Now, let $\omega$ be a $2$-form such that
$$m^{\CC,(L_e)}_\Gamma=0\Rightarrow m_\Gamma^{\CC,\omega,e_0}=0,$$
where $m^{\CC,(L_e)}$ denotes the complex multiplicity in the hereby considered enumerative problem, while $m_\Gamma^{\CC,\omega,e_0}$ denotes the complex multiplicity in the $\omega$-problem associated to $\omega$. See section \ref{section complex multiplicity} for a proper definition of the complex multiplicity. The choice $\omega=0$ always suits the hypothesis, but provides a trivial result. However, in many cases, it is possible to choose a non-trivial $\omega$, see example \ref{example wide class}. Under this technical hypothesis, we have the following invariance result.

\begin{theom}[\ref{theorem invariance general}]
There exists signs $\varepsilon_\Gamma$ only depending on the combinatorial type of $\Gamma$ and the constraints $(L_e)$ such that the count of solutions with multiplicity $\varepsilon_\Gamma B_\Gamma^{K_\omega}$ only depends on the choice of the slopes $L_e$ and not the affine subspaces $\mathcal{L}_e$.
\end{theom}

The obtained invariant is denoted by $\mathcal{B}_\Delta^{K_\omega}(L_e)$. The proof of the existence of the signs $\varepsilon_\Gamma$ is combined with the proof of invariance since the signs are in fact chosen so that the count of solutions remains invariant. Moreover, the signs $\varepsilon_\Gamma$ can be interpreted as orientations on the cones of the moduli space of tropical curves. With this interpretation, the multiplicities leading to invariant counts can be viewed as cycles in the \textit{star complex} of the moduli space. The star complex of a fan $\mathcal{F}$ is the CW-complex $\left(\mathcal{F}\backslash\{0\}\right)/\RR_+^*$. It can be seen as the space over which $\mathcal{F}$ is the cone. Such an interpretation may lead to further interesting invariants.

The idea that drives the proof is that with an additional technical assumption, a multiplicity that leads to an invariant in a specific enumerative problem may be used to get a new invariant in another enumerative problem.

We also prove a statement of continuity of the invariants $\mathcal{B}_\Delta^{K_\omega}(L_e)$ in terms of the slopes $(L_e)$. See Theorem \ref{theorem continuity general}.

Finally, we prove some enhancements of the main theorem \ref{theorem invariance general} by imposing conditions on marked points and not only unbounded ends, and study the possibility of replacing the affine constraints by tropical cycles.

\subsection{Motivations}

The study of the aforementioned tropical enumerative problems using the proposed refined multiplicity is justified by their connection to a family of classical enumerative problems, through the use of a correspondence theorem. We now explicit this connection.

In the planar case, the classical counterpart to the tropical enumerative problems are as follows: let $\CC\Delta$ be a toric surface endowed with its canonical real structure, the problem consists in finding the real rational curves in $\CC\Delta$ with fixed intersection with the toric boundary of $\CC\Delta$. This intersection comprises real points and pairs of complex conjugated points. This is the problem studied by Mikhalkin in \cite{mikhalkin2017quantum}. The general case takes place in a toric variety $\CC\Delta$, and consists in finding the real rational curves whose intersection points with the toric boundary belong to some chosen suborbits of the toric variety.

As mentioned above, in the planar case, Mikhalkin proved that there exists a refined signed count of the solutions to each classical enumerative problem that leads to an invariant. The refinement is provided by the value of quantum index, which up to a shift corresponds to the log-area of the oriented curves solutions to the problem. Recall that the log-area defined by the integral of the volume form of $N_\RR\simeq\RR^2$ pull-backed by the logarithmic map $\mathrm{Log}:N\otimes\CC^*\subset\CC\Delta\longrightarrow N_\RR$. For more details, see \cite{mikhalkin2017quantum,blomme2020computation}. Preferably to the Block-G\"ottsche multiplicity $m_\Gamma^{BG}$, the refined multiplicity $B_\Gamma^{K_\mathrm{det}}=\prod_V (q^{m_V}-q^{-m_V})$ is the one involved in the correspondence result of Mikhalkin in \cite{mikhalkin2017quantum} that relates the tropical refined invariant to the classical refined invariant. As a matter of fact, denoting by $\mathcal{N}_\Delta$ the classical refined invariant introduced in the same paper in case all intersection points are real, one has 
$$\mathcal{N}_\Delta=\mathcal{B}_\Delta^{\omega,e_0}.$$
In other terms, the multiplicity furnished by the correspondence theorem to count a tropical solution $\Gamma$, which corresponds to the signed refined number of solutions to the classical enumerative problem which tropicalize to $\Gamma$, is $B_\Gamma^{K_\mathrm{det}}$. This is no surprise since through the use of the correspondence theorem, the product over the vertices of the quantity $q^{a_V\wedge b_V}-q^{-a_V\wedge b_V}$ carries a deep connection to the set of all families of real oriented curves tropicalizing to $\Gamma$. This connection stays in higher dimension.

In higher dimension, although some signed counts for real curves are known, for instance by Welschinger \cite{welschinger2005enumerative}, there are at the time of writing no generalization for the refined signed count of Mikhalkin. A desirable generalization would be to make a signed count of the solutions to the aforementioned classical enumerative problem according to the value of a quantum class for which we now sketch a definition. For a toric variety $\CC\Delta$ with dense torus $N\otimes\CC^*$, one still has a logarithmic map $\mathrm{Log}:N\otimes\CC^*\longrightarrow N_\RR$. However, there is no canonical choice of a $2$-form to pull-back in order to measure the area of an oriented curve (except maybe $\omega$ in the case of the $\omega$-problem). For an oriented real curve $\varphi:S\dashrightarrow\CC\Delta$, the map that associates to a $2$-form $\varpi\in\Lambda^2 M_\RR$ the log-area $\int_{\varphi(S)}\mathrm{Log}^*\varpi$ of the curve under $\varpi$ is linear, and up to a shift corresponds to an element of $\mathrm{Hom}(\Lambda^2 M,\ZZ)\simeq\Lambda^2 N$. This element is called the \textit{quantum class} and is introduced in \cite{blomme2020phd}, to which we refer for more details. It is a generalization in higher dimension of the quantum index introduced by Mikhalkin in \cite{mikhalkin2017quantum} and shares many common properties with it. In particular it can be computed close to the tropical limit as a sum over the vertices of the limit tropical curve. The refined multiplicity of a tropical curve thus appears as a signed count of the solutions close to the tropical limit, refined by the value of their quantum class.

The existence of tropical invariants in the tropical enumerative problems considered in the paper proves the existence of a refined signed count of classical curves close to the tropical limit using this quantum class. Moreover, it suggests the existence of classical refined invariants generalizing the refined invariants from \cite{mikhalkin2017quantum} in the general case. As the definition of the signs in the classical setting is currently unknown, this papers only deals with the tropical side of the invariants, and will hopefully be connected to some classical refined invariant in a future work.

\subsection{Plan of the paper}

In the second section we gather the notions about tropical curves, their moduli spaces and the definition of the various involved multiplicities. The proofs for the computation of the complex multiplicities and the invariance statement it leads to are postponed to the appendices. Next, we introduce the $\omega$-\textit{problem}, a specific kind of enumerative problem in higher dimension for which we prove an invariance statement. This result is used in the fourth section to provide a new invariance result in the case of larger class of tropical enumerative problems. In the fifth section we study the first properties of the new refined invariants, more precisely their dependence on the constraints of the enumerative problems, and their possible enhancements. Last but not least, we give some examples.

\section{Tropical forest}
\label{section tropical curves}

\subsection{Abstract tropical curves}

Let $\overline{\Gamma}$ be a finite connected graph without bivalent vertices. Let $\overline{\Gamma}_\infty^0$ be the set of $1$-valent vertices of $\overline{\Gamma}$, and $\Gamma=\overline{\Gamma}\backslash\overline{\Gamma}^0_\infty$. We denote by $\Gamma^0$ the set of vertices of $\Gamma$, and by $\Gamma^1$ the set of edges of $\Gamma$. The non-compact edges resulting from the eviction of $1$-valent vertices are called \textit{unbounded ends}. The set of unbounded ends is denoted by $\Gamma^1_\infty$, while its complement, the set of bounded edges, is denoted by $\Gamma_b^1$. An element of $\Gamma^1_\infty$ is usually denoted by $e$.  Let $l:\gamma\in\Gamma_b^1\mapsto |\gamma|\in\RR_+^*=]0;+\infty[$ be a function, called length function. It endows $\Gamma$ with the structure of a metric graph by decreting that a bounded edge $\gamma$ is isometric to $[0;|\gamma|]$, and an unbounded end is isometric to $[0;+\infty[$.

\begin{defi}
Such a metric graph $\Gamma$ is called an \textit{abstract tropical curve}.
\end{defi}

An isomorphism between two abstract tropical curves $\Gamma$ and $\Gamma'$ is an isometry $\Gamma\rightarrow\Gamma'$. In particular an automorphism of $\Gamma$ does not necessarily respect the labeling of the unbounded ends since it only respects the metric. Therefore, an automorphism of $\Gamma$ induces a permutation of the set $\Gamma^1_\infty$ of unbounded ends.

An abstract tropical curve is said to be \textit{rational} if the first betti number $b_1(\Gamma)$ of the underlying graph, also called \textit{genus} of the graph, is equal to $0$. In graph theory, it is also called a \textit{tree}.

\subsection{Parametrized tropical curves}

Let $N$ be a lattice of rank $r\geqslant 2$, and $M=\mathrm{Hom}(N,\ZZ)$ be the dual lattice, which has the same rank. We denote by $N_\RR=N\otimes_\ZZ\RR$ the associated vector space, isomorphic to $\RR^r$. We now consider tropical curves inside $N_\RR$.

\begin{defi}
A parametrized tropical curve in $N_\RR$ is a pair $(\Gamma,h)$, where $\Gamma$ is an abstract tropical curve and $h:\Gamma\rightarrow N_\RR$ is a map such that:
\begin{itemize}
\item the restriction of $h$ to each edge of $\Gamma$ is affine with slope in $N$,
\item at each vertex $V\in\Gamma^0$, one has the \textit{balancing condition}: let $e_1,\dots,e_k$, be the edges adjacent to $V$, and let $u_1,\dots,u_k\in N$ be the outgoing slope of $h$ on them respectively, then one has
$$\sum_1^k u_i=0\in N.$$
\end{itemize}
Two parametrized tropical curves $(\Gamma,h)$ and $(\Gamma',h')$ are isomorphic if it exists an isometry $\varphi:\Gamma\rightarrow\Gamma'$ such that $h=h'\circ\varphi$. A parametrized tropical curve is \textit{rational} if the underlying abstract tropical curve is itself rational.
\end{defi}

For a parametrized tropical curve $h:\Gamma\rightarrow N_\RR$, the set of slopes of $h$ on the unbounded ends is called the \textit{degree} of the curve and is usually denoted by $\Delta$. It is a multiset of elements of $N$, and due to balancing condition, one has
$$\sum_{e\in\Gamma^1_\infty}n_e=0.$$

An abstract tropical curve $\Gamma$ is said to be \textit{trivalent} if every vertex $V\in\Gamma^0$ is adjacent to precisely three edges. Let $h:\Gamma\rightarrow N_\RR$ be a parametrized tropical curve with $\Gamma$ trivalent. For a vertex $V$, let $a_V,b_V,c_V$ be the outgoing slopes of $h$, satisfying $a_V+b_V+c_V=0$. Then we define the associated bivector
$$\pi_V=a_V\wedge b_V=b_V\wedge c_V=c_V\wedge a_V\in\Lambda^2 N,$$
defined up to sign.

If a tropical curve $\Gamma$ is not trivalent, we define the \textit{overvalence} of the graph to be
$$\mathrm{ov}(\Gamma)=\sum_{V\in\Gamma^0}\mathrm{val}(V)-3,$$
where $\mathrm{val}(V)$ is the valency of a vertex. The overvalence measures the default of trivalency of the graph.

\begin{rem}
If $N$ is of rank $2$, \textit{i.e.} $r=2$, then $\pi_V$ is a multiple of a generator of $\Lambda^2 N$, which is a lattice of rank $1$. The absolute value of the scalar is called \textit{vertex multiplicity}.
\end{rem}

\subsection{Moduli space of rational tropical curves}

We now focus on the structure of the \textit{moduli spaces} of rational tropical curves, both abstract and parametrized. These spaces serve as parameter spaces for tropical curves. They are the domain of some suitable \textit{composed evaluation maps} allowing us to work conveniently with the enumerative problems considered in the following sections.

\begin{defi}
Let $\Gamma$ be an abstract rational tropical curve. The \textit{combinatorial type} of $\Gamma$, denoted by $\mathrm{Comb}(\Gamma)$ is the homeomorphism type of the underlying graph, \textit{i.e.} the labeled graph without the metric.
\end{defi}

Conversely, to give a graph a tropical structure, one just needs to assign a length to the bounded edges of the graph. If the curve is \textit{trivalent}, there are precisely $|\Gamma_\infty^1|-3$ bounded edges. Otherwise, there are $|\Gamma_\infty^1|-3-\mathrm{ov}(\Gamma)$ bounded edges, where $\mathrm{ov}(\Gamma)$ is the \textit{overvalence} of the graph. The set of tropical curves having the same combinatorial type is then parametrized by $\RR_+^{|\Gamma_\infty^1|-3-\mathrm{ov}(\Gamma)}$, where the coordinates are given by the lengths of the bounded edges.\\

For a given combinatorial type $\mathrm{Comb}(\Gamma)$, the boundary of the orthant $\RR_+^{|\Gamma_\infty^1|-3-\mathrm{ov}(\Gamma)}$ corresponds to tropical curves for which at least one edge has length $0$. Those tropical curves can be identified with curves whose underlying graph has been suitably modified: edges with zero length are deleted and their extremities are merged together. These identifications allow us to glue together the cones corresponding to the finitely many combinatorial types of graphs. One obtains the \textit{moduli space} $\mathcal{M}_{0,|\Gamma_\infty^1|}$ of rational curves with $|\Gamma_\infty^1|$ marked points which has thus the structure of a fan.

The space $\mathcal{M}_{0,m}$ is a simplicial fan of pure dimension $m-3$. The top-dimensional cones parametrize curves whose underlying graph is trivalent. Codimension 1 cones, which are called \textit{facets}, correspond to combinatorial types whose underlying graph is trivalent except for one vertex which is quadrivalent. Each facet belongs to exactly three faces, which correspond to the three possible \textit{smoothings} of the quadrivalent vertex into a pair of trivalent vertices.

Given an abstract rational tropical curve $\Gamma$, and a degree $\Delta\subset N$ assigning a slope $n_e\in N$ to each unbounded end of $\Gamma$, there is a unique way up to translation to define a map $h:\Gamma\rightarrow N_\RR$ which makes $(\Gamma,h)$ into a parametrized tropical curve of degree $\Delta$: the slope of the unbounded ends is fixed by the degree, and the slope of interior edges is fixed by the balancing condition. The translation is fixed by the position of a vertex in $\Gamma$. This observation gives a bijection between the moduli space $\modpar$ of parametrized rational tropical curves of degree $\Delta$ in $N_\RR$ and $\mathcal{M}_{0,|\Delta|}\times N_\RR$, where the bijection is given by the projection to the underlying graph and the position of a chosen vertex, for instance the vertex adjacent to a fixed unbounded end $e_0$.

The cones of the moduli space $\modpar$ are of the form $\RR_+^{|\Delta|-3}\times N_\RR$, whose tangent space contains the lattice $\ZZ^{|\Delta|-3}\times N$.

\begin{rem}
The multiset $\Delta$ is allowed to contain the $0$ vector. Unbounded ends on which the map $h$ has zero slope are sometimes called \textit{marked points} because their image inside $N_\RR$ is a contracted to a point.
\end{rem}

\begin{rem}
Notice that the set of unbounded ends of parametrized tropical curves in $\modpar$ are canonically identified since they are labeled. Thus, one can speak about $\Gamma_\infty^1$ without having to specify a curve $\Gamma$. The degree $\Delta$ is then a map $e\in\Gamma_\infty^1\mapsto n_e\in N$.
\end{rem}

The set of combinatorial types of parametrized tropical curves of degree $\Delta$ is denoted $\mathcal{C}_\Delta$, where the $\mathcal{C}$ stands both for combinatorial and for cones. Among them, the set of top-dimensional combinatorial types is denoted by $\mathcal{C}_\Delta^\mathrm{top}$. Otherwise, the set of cones of a given dimension is denoted by $\mathcal{C}_\Delta^{(k)}$. A combinatorial type is usually denoted by the letter $C$, whose notation is also used to denote the associated cone in the moduli space $\modpar$. Moreover, for a given parametrized tropical curve $h:\Gamma\rightarrow N_\RR$, we sometimes forget the $h$ in the notation, and use the same notation $\Gamma$ for the curve and the combinatorial type $\mathrm{Comb}(\Gamma)$ that it belongs to.

At some point, we need to orient the cones of the moduli space $\modpar$, which means find an orientation for the underlying lattice $\ZZ^{|\Delta|-3}\times N$ of each cone $C\in\mathcal{C}_\Delta^\mathrm{top}$, or similarly, its linear span. If $C\in\mathcal{C}_\Delta^\mathrm{top}$ is a combinatorial type, associated to an orthant of the moduli space $\modpar$ and denoted by the same letter, we denote by $\ori(C)$ an orientation of the linear span of the cone. If $W$ is a facet adjacent to $C$, $W$ belongs to the boundary of $C$. Then, as the normal bundle of $W$ in $C$ is canonically oriented, an orientation $\ori(C)$ induces an orientation of the tangent bundle to $W$, which is denoted by $\partial_W\ori(C)$. Notice that a reversing of the orientation of $C$ also reverses the induced orientation on $W$.

\subsection{Evaluation map}

For each unbounded end $e\in\Gamma^1_\infty$, let $n_e$ be the slope assigned by the degree $\Delta$. One has a well-defined \textit{evaluation map} on the moduli space $\modpar$, which sends a parametrized tropical curve to the image of the unbounded end inside the quotient $N_\RR/\langle n_e\rangle$. If the slope is $0$, this quotient is just $N_\RR$ itself and the image is a point, justifying the name \textit{marked point}. Otherwise, the image corresponds to the line that contains the image of the unbounded end under $h$. Putting all these maps together, we get the \textit{evaluation map}:
$$\begin{array}{cccl}
\mathrm{ev}: & \mathcal{M}_{0}(\Delta,N_\RR) & \longrightarrow & \prod_{e\in\Gamma^1_\infty} N_\RR/\langle n_e\rangle, \\
 & (\Gamma,h) & \longmapsto & \left( h(p_e)\right)_e \\
\end{array}$$
where $p_e$ is any point chosen in $e\in\Gamma^1_\infty$.\\

By definition, the evaluation map is linear on every cone of $\modpar\simeq\mathcal{M}_{0,|\Delta|}\times N_\RR$ but is not necessarily injective on top-dimensional cones. For instance, it is not the case when some bounded edge has zero slope, or when the combinatorial type contains a flat vertex, which is a vertex whose adjacent edges are contained in a line.

\subsection{Complex multiplicity of tropical curves}
\label{section complex multiplicity}

One main feature of tropical enumerative geometry is to count tropical curves solution to a given enumerative problem $\mathcal{P}$, which means finding the curves which are subject to some constraints. It is often necessary to use suitable multiplicities $m$ to get a count that does not depend on the choice of the constraints when those vary in some family. Such a count using multiplicity $m$ is said to lead to an \textit{invariant} for the problem $\mathcal{P}$. Here, we describe the main multiplicities that are of interest to us, procrastinating the proper introduction to the associated enumerative problems.

\begin{rem}
In some particular enumerative problems, some multiplicities leading to an invariant are provided through the use of some correspondence theorems, such as Mikhalkin's \cite{mikhalkin2005enumerative}, Shustin's \cite{shustin2002patchworking}, Nishinou and Siebert's \cite{nishinou2006toric} or Tyomkin's \cite{tyomkin2017enumeration}. The use of these multiplicities leads to tropical invariants which have the value of some classical invariants, \textit{e.g.} Gromov-Witten invariants, Welschinger invariants, or refined invariants introduced by Mikhalkin in \cite{mikhalkin2017quantum}. Such a relation to classical invariants is not systematic although is sometimes suspected, as it is the case for refined invariants introduced by Block and G\"ottsche \cite{block2016refined}.
\end{rem}

In this section, we recall the definition of the \textit{complex multiplicity}, which is the multiplicity that comes from the correspondence theorem \cite{mikhalkin2005enumerative} in the planar case, or \cite{nishinou2006toric} in the higher dimensional case. Let $\Gamma$ be a parametrized tropical curve of degree $\Delta\subset N$. For each unbounded end $e$, let $L_e$ be a full sublattice of $N/\langle n_e\rangle$, meaning that the quotient is without torsion. In order to get small notations, we denote by $Q_e$ the quotient of $N/\langle n_e\rangle$ by $L_e$, which is a lattice of rank $\mathrm{rk}\ N/\langle n_e\rangle - \mathrm{rk}\ L_e=\mathrm{cork}\ L_e$. We denote by $Q_e^\RR=Q_e\otimes_\ZZ \RR$ the associated vector space. We assume that $\sum_{e\in\Gamma^1_\infty} \mathrm{cork}L_e$ is equal to the dimension of $\modpar$, which is $|\Delta|+r-3$. We then compose the evaluation map with the quotient maps $N_\RR/\langle n_e\rangle\rightarrow Q_e$, getting what we call a \textit{composed evaluation map}:

$$\mathrm{ev}^{(L_e)}:\modpar\rightarrow\prod_{e\in\Gamma^1_\infty} Q_e^\RR.$$

For a combinatorial type $C\in\mathcal{C}_\Delta^\mathrm{top}$, we denote by $\mathrm{ev}^{(L_e)}_C$ the restriction to the orthant of the combinatorial type $C$, which is linear. Due to the assumption on the ranks of the lattices, both spaces have the same dimension.

\begin{defi}
We set the \textit{complex multiplicity} to be
$$m^{\CC,(L_e)}_\Gamma = |\mathrm{det} \ \mathrm{ev}^{(L_e)}_C|,$$
where the domain is endowed with a basis of the lattice $\ZZ^{|\Delta|-3}\times N$, and the codomain a basis of the lattice $\prod_e Q_e$.
\end{defi}

\begin{rem}
One other way to describe the complex multiplicity is as follows: consider the sublattice $\prod_e L_e\subset \prod_e N/\langle n_e\rangle$, which is a full sublattice of codimension $|\Delta|+r-3$ by assumption. Let $\alpha_{(L_e)}\in\Lambda^{|\Delta|+r-3}\left(\prod_e N/\langle n_e\rangle\right)^*$ be its Pl\"ucker vector obtain as follows: take the wedge product of a basis of the dual orthogonal of $\prod_e L_e$. Then the pull-back of $\alpha_{(L_e)}$ by the evaluation map $\mathrm{ev}_C$ restricted to a cone $C$ is a top-dimensional form. Thus, it is a scalar multiple of the determinant form on the lattice of the orthant. The absolute value of the scalar factor is the multiplicity.

It also follows from standard tropical intersection theory (see \cite{allermann2010first}) that the complex multiplicity corresponds to the index of the lattice spanned by $\mathrm{ev}_C(\ZZ^{|\Delta+r-3}\times N)$ and $\prod_e L_e$ inside $\prod_e N/\langle n_e\rangle$, which are of complementary dimension if $\mathrm{ev}$ is injective on the cone relative to the combinatorial type. If not, the complex multiplicity is $0$.\\
\end{rem}

The complex multiplicity can be computed in the hereby described fancy way. This can be seen as a particular case of a result of T.~Mandel and H.~Ruddat \cite{mandel2019tropical}, which we restate in our setting for sake of completeness. To keep a general setting, here and only here, we consider tropical curves $(\Gamma,h)$ that are not necessarily trivalent. Nevertheless, for such a curve and an evaluation map
$$\mathrm{ev}^{(L_e)}:\mathcal{M}_0(\Delta,N_\RR)\rightarrow\prod_{e\in\Gamma^1_\infty} Q_e^\RR,$$
the complex multiplicity is still defined as the determinant of $\mathrm{ev}^{(L_e)}$ restricted to the cone of the combinatorial type, provided that it has the same dimension as the target space. We assume it is the case.\\

For each primitive sublattice $L\subset N$ of corank $l$, there is an orthogonal dual primitive sublattice $L^\perp\subset M$ of rank $l$, and an associated Pl\"ucker vector $\rho\in\Lambda^l M$ defined up to sign. This polyvector is defined as follows: let $m_1,\dots,m_l$ be a basis of $L^\perp$. Then, one has
$$\rho=m_1\wedge\cdots\wedge m_l \in\Lambda^l M.$$
This polyvector does not depend on the chosen basis up to sign. Hence, given a parametrized rational tropical curve $h:\Gamma\rightarrow N_\RR$, one has a polyvector $\rho_e$ associated to each of the unbounded ends $e$ of the curve: the Pl\"ucker vector associated to the full sublattice spanned by $n_e$ and $L_e$.\\

We choose one vertex of $\Gamma$ to be the sink of the curve and orient every edge toward it. We then cut the branches of the tropical curve with the following rule: let $V$ be a vertex different from the sink, with incoming unbounded ends directed by $n_1,\dots,n_s$, and respective polyvectors $\rho_1,\dots,\rho_s$, and a unique outgoing bounded edge, thus directed by $n_1+\cdots +n_s$. The polyvector associated to the outgoing edge of $V$ is
$$\rho=\iota_{n_1+\cdots+n_s}(\rho_1\wedge\cdots\wedge\rho_s).$$
Recall that $\iota_n$, for $n\in N$, denotes the interior product by $n$. Geometrically, the polyvector associated to an edge and to an unbounded end is a multiple of the Pl\"ucker vector associated to the space described by the edge when it moves. This means the following:
\begin{itemize}[label=-]
\item For an unbounded end $e\in\Gamma^1_\infty$ associated to a sublattice $L_e$, the unbounded end can move in the direction $L_e$, thus describe the space with slope spanned by $n_e$ and $L_e$.
\item At a vertex $V$ with incoming edges having respective polyvectors $\rho_1,\dots,\rho_s$. Assume by induction that the incoming edges move in an affine space directed by a vector subspace whose Pl\"ucker vector is $\rho_i$ respectively. Then, the vertex moves in the intersection of all these subspaces. Therefore, the affine subspace has Pl\"ucker vector $\rho_1\wedge\cdots\wedge\rho_s$.
\item Finally, for the outgoing edge of $V$, it moves in an affine space equal to the affine space where $V$ lives, enlarged by the direction of the edge: $n_1+\cdots+n_s$. Hence, the polyvector is obtained by making the interior product with $n_1+\cdots+n_s$.
\end{itemize}

At the sink, let $\rho_1,\dots,\rho_s$ be the polyvectors associated to the incoming adjacent edges. Because of the assumption on dimensions, we have
$$\rho_1\wedge\cdots\wedge\rho_s\in\Lambda^r M\simeq\ZZ.$$
Thus, it is an integer multiple of a generator of $\Lambda^r M$. The absolute value of the proportionality constant, obtained by evaluating on a basis of $N$, is the desired determinant. The proof that it does not depend on the choice of the sink and that it computes the determinant is proved in the appendix.

\begin{theom}[\ref{Theorem calcul complex multiplicity tropical curve}]
The value obtained by the preceding algorithm is equal to the complex multiplicity $m_\Gamma^\CC$.
\end{theom}

\subsection{Refined multiplicities of a trivalent tropical curve}
\label{section Refined multiplicities of a trivalent tropical curve}

To finish this section about tropical curves, we define the refined multiplicities that we use in the following sections. Let $\mathcal{V}_\Delta$ be the set of types of trivalent vertices that occur in some combinatorial type of $\modpar$. A trivalent vertex is determined by the slopes of the three outgoing edges. Using the balancing condition, the type of a vertex is fully determined by the partition of $\Delta$ into three subsets that it induces. Hence, $\mathcal{V}_\Delta$ is finite.

Let $\omega_0$ be a $2$-form such that for any $V\in\mathcal{V}_\Delta$ with $\pi_V\neq 0$, $\omega_0(\pi_V)\neq 0$. This condition means that for any non-flat vertex, the restriction of $\omega_0$ to the plane spanned by the outgoing edges is non-zero. The \textit{refined multiplicity} is defined as follows.

\begin{defi}
The refined multiplicity of a parametrized tropical curve $h:\Gamma\rightarrow N_\RR$ is
$$B^{0}_\Gamma=\prod_{V\in\Gamma^0} (q^{\pi_V}-q^{-\pi_V})\in\ZZ[\Lambda^2 N],$$
where the bivector $\pi_V$ of $V$ is chosen such that $\omega_0(\pi_V)>0$. If the curve contains some flat vertex $W$, then $\pi_W=0$ and then $B^0_\Gamma=0$.
\end{defi}

The refined multiplicity is a Laurent polynomial in several variables. The number of indeterminates is equal to the rank of $\Lambda^2 N$, which is $\bino{r}{2}$. It depends up to sign on the choice of $\omega_0$.\\

Let $K\subset\Lambda^2 N$ be a sublattice. Then we can compose the definition of $B^0_\Gamma$ with the projection $p_K:\ZZ[\Lambda^2 N]\longrightarrow \ZZ[\Lambda^2 N/K]$ to get a slightly less refined multiplicity.

\begin{defi}
For a sublattice $K\subset\Lambda^2 N$, let $B_\Gamma^{K}=p_K(B^0_\Gamma)$, obtained by reducing the exponents modulo $K$.
\end{defi}

Notice that this notation is consistent with the notation for $B_\Gamma^0$, which corresponds to the choice $K=\{0\}$. Although the relation $p_K(B^0_\Gamma)=B^K_\Gamma$ shows that $B^0_\Gamma$ determines the other refined multiplicities, the use of these multiplicities is necessary since the multiplicity $B_\Gamma^0$ might fail to provide an invariant, while some $B_\Gamma^K$ might do so.

\begin{rem}
If $N$ is of rank $2$, then $\Lambda^2 N\simeq \ZZ$, and the choice of a generator of $\Lambda^2 N$ identifies $\ZZ[\Lambda^2 N]$ with $\ZZ[q,q^{-1}]$. Up to a division by $(q-q^{-1})^{|\Delta|-2}$, we recover the definition of the refined multiplicity proposed by Block and G\"ottsche in \cite{block2016refined}.
\end{rem}

\begin{rem}
The refined multiplicity depends up to a sign on the choice of $\omega_0$. In every occurrence of the refined multiplicity, the choice of $\omega_0$ is fixed in a non-ambiguous way so it does not bring up any problem.
\end{rem}

\section{$\omega$-problem and first refinement}

Before turning to more general enumerative problems and providing them with a refined way to count the solutions, we present a family of enumerative problems, called $\omega$-\textit{problems}, where it is slightly easy to show the existence of an associated refined invariant, \textit{i.e.} a refined way to count the solutions of the problem and leading to a result independent of the constraints.\\

In the planar case, meaning that $r=2$, one can consider the classic enumerative problem of finding degree $\Delta$ rational tropical curves passing through $|\Delta|-1$ points. In that case, the complex multiplicities of rational tropical curves split themselves as a product over the vertices of the curve. This fact, which is specific to the two-dimensional case and fails in arbitrary dimension, is used to prove invariance with the refined multiplicity of Block-G\"ottsche \cite{block2016refined}, see \cite{itenberg2013block}. The $\omega$-\textit{problems} also share this property of the complex multiplicity splitting into a product over the vertices of the curve.

\begin{rem}
In the planar case, the $\omega$-problem amounts to find rational tropical curves passing through $|\Delta|-1$ points when these points are located on the toric divisors, meaning that the unbounded ends are contained in fixed lines of the same slope.
\end{rem}

\subsection{Enumerative problem}

Let $\Delta\subset N$ be a tropical degree consisting only of non-zero vectors, and let $\omega\in\Lambda^2 M$ be a $2$-form chosen such that $\omega(n_e,-)\neq 0$ for every $n_e\in \Delta$. Before getting to the enumerative problem, let us define what we mean by \textit{moment}.

\begin{defi}
For a parametrized tropical curve $h:\Gamma\rightarrow N_\RR$ of degree $\Delta$, the \textit{moment} $\mu_\omega(e)$ of an unbounded end $e\in\Gamma^1_\infty$ with respect to $\omega$ is defined as follows: let $n_e$ be the slope of $h$ on $e$ and let $p$ be any point in $h(e)$,
$$\mu_\omega(e)=\omega(n_e,p).$$
\end{defi}

This definition may be extended to the moment of a bounded edge of $\Gamma$ if an orientation is specified. Geometrically, the moment corresponds to the position of the affine space $\mathcal{H}_e$ of slope $H_e=\ker\iota_{n_e}\omega$ that contains the image unbounded end $e$. Using the balancing condition at each vertex of the curve, one easily proves the following proposition.

\begin{prop}
For $h:\Gamma\rightarrow N_\RR$ a parametrized tropical curve, one has
$$\sum_{e\in\Gamma^1_\infty}\mu_\omega(e)=0,$$
meaning that the family $\mu$ of moments of the curve belongs to $\mathfrak{M}=\{\mu:\ \sum_e\mu_e=0\}\subset\RR^{\Delta}$.
\end{prop}

\begin{rem}
This fact generalizes the planar Menelaus relation noticed by Mikhalkin in \cite{mikhalkin2017quantum}. That is why we also call it \textit{tropical Menelaus relation}.
\end{rem}

We now state the $\omega$-\textit{problem}. Consider the global moment map
$$\mathrm{mom}:(\Gamma,h)\in\modpar\longmapsto \left(\mu_\omega(e)\right)_e\in \mathfrak{M}\subset\RR^{\Delta}.$$
The moment map, which has image in $\mathfrak{M}$ due to the tropical Menelaus relation, is a specific case of composed evaluation map, where one chooses $L_e=H_e$.

\begin{prob}
Let $\mu\in\mathcal{H}$. How many parametrized rational tropical curves of degree $\Delta$ satisfy $\mathrm{mom}(\Gamma,h)=\mu$ ? In other words, for each $e\in\Gamma^1_\infty\backslash\{e_0\}$, fix $\mathcal{H}_e$ an affine hyperplane with slope $H_e$. How many parametrized rational tropical curves of degree $\Delta$ satisfy $h(e)\subset \mathcal{H}_e$ ?
\end{prob}

Unfortunately, the codomain has dimension $|\Delta|-1$, while the domain has dimension $|\Delta|+r-3$. Therefore, unless $r=2$, this is not enough to ensure a finite number of solutions. We then add the following condition: let $e_0\in\Gamma_\infty^1$ be a specific unbounded end and let $\delta\in N/\langle n_{e_0}\rangle$ be a primitive vector such that $\omega(n_{e_0},\delta)\neq 0$ and $D$ be a line with slope $\delta$. We add the condition that $h(e_0)\in D\subset N_\RR/\langle n_{e_0}\rangle$, getting to the $\omega$-\textit{problem} denoted by $\mathcal{P}(\omega,e_0)$.

\begin{prob}
\label{problem moment line}
Let $\mu\in\mathcal{H}$ and $D$ be a line with slope $\delta$ inside $N_\RR/\langle n_{e_0}\rangle$. How many parametrized rational tropical curves of degree $\Delta$ satisfy $\mathrm{mom}(\Gamma,h)=\mu$ and $h(e_0)\in D$ ?
\end{prob}

\begin{rem}
The condition $h(e_0)\in D$ means that the image of the unbounded end $e_0$ is contained in a plane.
\end{rem}

We denote the associated composed evaluation map as follows:
$$\mathrm{ev}^{\omega,e_0}:\modpar\longrightarrow N_\RR/\gen{\delta}\times\prod_{e\neq e_0}\RR,$$
where the $\RR$ terms correspond to the evaluation of the moments of the unbounded ends different from $e_0$.

\begin{rem}
We omit $\delta$ from the notation $\mathcal{P}(\omega,e_0)$, since, as proven in Theorem \ref{theorem invariance moment problem}, the solutions do not depend on $\delta$.
\end{rem}

\begin{rem}
There are other ways to impose further constraints on the curves to get a finite number of tropical curves:
\begin{itemize}
\item The problem depends on the chosen end $e_0$.
\item The constraint could have been split into two constraints on two different unbounded ends.
\item One could have impose constraints on one or several new marked points on the curve.
\end{itemize}
Some of these other possibilities are studied later on in section \ref{section first properties and possible}.
\end{rem}

\subsection{Count of solutions and complex invariance statement}

We now get to the invariance statement for the count of solutions to $\mathcal{P}(\omega,e_0)$ using first the complex multiplicity, meaning that the count does not depend on the choice of $\mu$ nor $D$ as long as it is generic. The invariance of the count using complex multiplicity has already been proven in various ways: by hand showing the local invariance around walls, using tropical intersection theory, or using the invariance in the complex setting and some correspondence theorem. In any case, the statement is a particular case of a more general statement which is proved in the appendix. Yet, as the invariance with the complex multiplicity is a key ingredient in the proof of invariance using refined multiplicity, we give here the main steps of what could be called the \textit{pedestrian} proof.

\begin{prop}
For a generic choice of $\mu$ and $D$, there are finitely many tropical curves solution to $\mathcal{P}(\omega,e_0)$. Moreover, such curves are trivalent.
\end{prop}

\begin{proof}
The set of top-dimensional combinatorial types $\mathcal{C}_\Delta^\mathrm{top}$ decomposes between the types on which the composed evaluation map $\mathrm{ev}^{\omega,e_0}$ is non-injective, and the ones where it is bijective. The image of the cones where the composed evaluation map is non-injective is nowhere dense inside the codomain since $\mathrm{ev}^{\omega,e_0}$ is linear and non-surjective either. The same applies for the image of the cones which are not top-dimensional. Therefore, if $(\mu,D)$ is chosen outside the image of these cones, the only cones that may provide a solution are top-dimensional ones where the evaluation map is injective. Such a solution, if it exists, is unique. As there are finitely many combinatorial types, the conclusion follows.
\end{proof}

\begin{rem}
This fact can also be proven using tropical intersection theory: for a generic choice of constraints, the spaces $\mathrm{ev}(\modpar)$ and $D\times\prod_{e\neq e_0} \mathcal{H}_e$ intersect transversally inside $\prod_e N_\RR/\gen{n_e}$ and intersection points belong to the interior of top-dimensional faces. The conclusion follows along with the invariance of the count using complex multiplicity.
\end{rem}

We now compute the complex multiplicities of the solutions.

\begin{lem}
For each trivalent solution $h:\Gamma\rightarrow N_\RR$ to the $\omega$-problem, the complex multiplicity takes the following form:
$$m_\Gamma^{\CC,\omega,e_0}=\left| \omega(n_{e_0},\delta)\prod_{V\in\Gamma^0} \omega(\pi_V)\right|.$$
\end{lem}

\begin{proof}
We use the algorithm to compute the complex multiplicity by choosing as sink the vertex adjacent to the unbounded end $e_0$. Then, the recursion uses the following identity:
\begin{align*}
\iota_{a+b}(\iota_a\omega\wedge\iota_b\omega) &=(\iota_{a+b}\iota_a\omega)\iota_b\omega - \iota_a\omega(\iota_{a+b}\iota_b\omega) \\
&= \omega(a,a+b)\iota_b\omega - \omega(b,a+b)\iota_a\omega \\
&= \omega(a,b)\iota_{a+b}\omega.\\
\end{align*}
Hence, the computation reduces to the case there is just one vertex. Noticing that the constraint on $e_0$ is associated to the polyvector $\iota_\delta\iota_{n_{e_0}}\Omega$, where $\Omega$ is a generator of $\Lambda^r M$, we have
\begin{align*}
\iota_a\omega\wedge\iota_b\omega\wedge\iota_\delta\iota_{a+b}\Omega & = \pm \omega(a+b,\delta)\omega(a,b)\Omega.
\end{align*}
\end{proof}

To conclude, for a generic choice $(\mu,D)$, let
$$N^\CC_\Delta(\mu,D)=\sum_{\substack{\mathrm{mom}(\Gamma,h)=\mu \\ h(e_0)\in D}} m^{\CC,\omega,e_0}_\Gamma,$$
be the count of solutions to $\mathcal{P}(\omega,e_0)$ using the complex multiplicity. This count is well-defined since the solutions $\Gamma$ are trivalent. Then one has the following invariance statement.

\begin{prop}
The value of the count $N^\CC_\Delta(\mu,D)$ does not depend on the choice of $\mu$ nor the choice of $D$ as long as these choices are generic. It only depends on the choice of $\omega$ and the slope $\delta$ of $D$.
\end{prop}

The proof follows from tropical intersection theory, or from the computations done in the appendix.

\subsection{Refined invariance statement}

We conclude this section by showing an invariance statement for the same enumerative problem \ref{problem moment line}, using this time the refined multiplicities $B^{K_\omega}$ introduced in section \ref{section Refined multiplicities of a trivalent tropical curve}.\\

Let $K_\omega=\langle\pi_V:\omega(\pi_V)=0\rangle_{V\in\mathcal{V}_\Delta}\subset\Lambda^2 N$ be the space spanned by the vertices spanning a plane where $\omega$ restricts to $0$. If $\omega$ is chosen generically regarding $\Delta$, then $K_\omega=0$. Otherwise, the $2$-form $\omega$ might take value $0$ on some vertex $V$ occurring in a rational tropical curve of degree $\Delta$. We use the refined multiplicity $B_\Gamma^{K_\omega}$ of a tropical curve, which is
$$B^{K_\omega}_\Gamma=\prod_{V\in\Gamma^0}(q^{\pi_V}-q^{-\pi_V})\in \ZZ[\Lambda^2 N/K_\omega].$$
The $2$-vector $\pi_V$ is chosen so that $\omega(\pi_V)>0$. If $\omega(\pi_V)=0$, then $\pi_V\in K_\omega$ and the multiplicity is equal to $0$ since the exponents are taken modulo $K_\omega$. Then, for a generic choice $(\mu,D)$, let
$$\mathcal{B}_\Delta(\mu,D)=\sum_{\substack{\mathrm{mom}(\Gamma,h)=\mu \\ h(e_0)\in D}} B^{K_\omega}_\Gamma\in \ZZ[\Lambda^2 N/K_\omega].$$

\begin{theo}
\label{theorem invariance moment problem}
The refined count $\mathcal{B}_\Delta(\mu,D)$ does not depend on the choice of $\mu$ and $D$ as long as it is generic. Neither does it depend on the slope $\delta$ of $D$. It only depends on the choice of $\omega$ and the choice of the unbounded end $e_0$ with the additional constraint. It is denoted by $\mathcal{B}_\Delta^{\omega,e_0}$.
\end{theo}

\begin{proof}
We already know that the count of solutions with complex multiplicity $m^{\CC,\omega,e_0}$ leads to an invariant. This means that the repartition of the solutions around a wall matches the invariance of the count with complex multiplicities. Recall that here, a wall is a tropical curve with a unique quadrivalent vertex. Therefore, we need to check for each wall that the count with refined multiplicities is also invariant.

We check the local invariance around the wall corresponding to a quadrivalent vertex of the tropical curve. Let $a_1,a_2,a_3,a_4$ be the slopes of the outer edges, with index taken in $\ZZ/4\ZZ$. Up to a relabeling, we assume that $\omega(a_i,a_{i+1})>0$ for every $i$. If some value is equal to $0$, the proof remains unchanged. Moreover, we assume that $\omega(a_2,a_3)>\omega(a_1,a_2)$. The invariance of the complex count amounts to the following relation:
$$\begin{array}{ccccccc}
	\omega(a_1,a_2)\omega(a_1+a_2,a_3) & + & \omega(a_1,a_3)\omega(a_2,a_1+a_3) & + & \omega(a_2,a_3)\omega(a_2+a_3,a_1) & = &0, \\
\text{ for }12//34 & & \text{ for }13//24 & & \text{ for }14//23 & & \\
	\end{array}$$
	and the repartition of combinatorial types around the wall is given by the sign of each term. It means up to sign that one is positive and is on one side of the wall, and the two other ones are negative, on the other side of the wall. Hence, we just need to study the signs of each term to know which curve is on which side. We know that $\omega(a_1,a_2)$ and $\omega(a_1+a_2,a_3)=\omega(a_3,a_4)$ are positive. Therefore, their product, which is the term of $12//34$, is also positive. We know that $\omega(a_2,a_3)$ is positive, but $\omega(a_2+a_3,a_1)=-\omega(a_4,a_1)$ is negative. Therefore, their product is negative and $14//23$ is on the other side of the wall. It means that the combinatorial types $12//34$ and $14//23$ are on opposite sides of the wall. We need to determine on which side the type $13//24$ is, and that is given by the sign of the middle term. As by assumption $\omega(a_2,a_1+a_3)=\omega(a_2,a_3)-\omega(a_1,a_2)>0$, it is determined by the sign of $\omega(a_1,a_3)$.\\
	\begin{itemize}
	\item If $\omega(a_1,a_3)>0$, then $12//34$ and $13//24$ are on the same side, and the invariance for refined multiplicities is dealt with the identity
	\begin{align*}
 & (q^{a_2\wedge a_3}-q^{a_3\wedge a_2})(q^{a_1\wedge (a_2+a_3)}-q^{(a_2+a_3)\wedge a_1}) \\
 = & (q^{a_1\wedge a_2}-q^{a_2\wedge a_1})(q^{(a_1+a_2)\wedge a_3}-q^{a_3\wedge (a_1+a_2)}) \\
+ & (q^{a_1\wedge a_3}-q^{a_3\wedge a_1})(q^{a_2\wedge (a_1+a_3)}-q^{(a_1+a_3)\wedge a_2}),\\
\end{align*}
	\item and if $\omega(a_1,a_3)<0$, then $14//23$ and $13//24$ are on the same side and then the invariance for refined multiplicities is true since

\begin{align*}
 & (q^{a_2\wedge a_3}-q^{a_3\wedge a_2})(q^{a_1\wedge (a_2+a_3)}-q^{(a_2+a_3)\wedge a_1}) \\
+ & (q^{a_3\wedge a_1}-q^{a_1\wedge a_3})(q^{a_2\wedge (a_1+a_3)}-q^{(a_1+a_3)\wedge a_2})\\
 = & (q^{a_1\wedge a_2}-q^{a_2\wedge a_1})(q^{(a_1+a_2)\wedge a_3}-q^{a_3\wedge (a_1+a_2)}). \\
\end{align*}

	\end{itemize}
	
	This closes the proof if each complex multiplicity is non-zero. If the multiplicity of some side of the wall is $0$, it does not provide any solution to the enumerative problem for generic values, and one needs to cancel its refined multiplicity, which might not be $0$ if one chooses the refined multiplicity $B^0_\Gamma$. If some combinatorial type has complex multiplicity $0$, it means that for some $W$ among the two displayed vertices, $\omega(\pi_W)=0$. The same relations between refined multiplicity lead to the invariance provided that one mods out the exponent by $\pi_W$, which sends the refined multiplicity to $0$. One just obtains that the curves on the two remaining sides of the wall have the same multiplicity.
	
	The result also does not depend on $\delta$ since the refined multiplicity does not depend on $\delta$, and the set of solutions for a given choice of $(\mu,D)$ only depends on $\mu$ and the intersection point of $D$ with the hyperplane $\mathcal{H}_{e_0}$.
\end{proof}

\begin{rem}
We have proven a refined invariance statement for every $\omega$ such that for every $n\in\Delta$, $\iota_n\omega\neq 0$. For every $\omega$ not taking the value $0$ on $\mathcal{V}_\Delta$, we have an refined invariant in $\ZZ[\Lambda^2 N]$. If for some $V$ one has $\omega(\pi_V)=0$, the space of the exponents where the refined invariant takes value is reduced by $\pi_V$.
\end{rem}

\section{Refined count for rational tropical curves}

We now use the refined invariance proven in the setting of the $\omega$-problem to prove a refined invariance for a larger class of enumerative problems, where the constraints we ask a tropical curve to satisfy do not necessarily come from the choice of a $2$-form $\omega$.

\subsection{New enumerative problems}

Let $\Delta\subset N$ be a tropical degree without any zero vector. For each $n_e\in\Delta$, let $L_e$ be a full sublattice of $N/\gen{n_e}$, and let $\mathcal{L}_e$ be an affine subspace of $N_\RR/\langle n_e\rangle$ with slope $L_e$. A \textit{full} sublattice means that the quotient is without torsion. Assume that $\sum_e \mathrm{cork}L_e=|\Delta|+r-3$.

\begin{prob}
How many rational tropical curves $h:\Gamma\rightarrow N_\RR$ of degree $\Delta$ satisfy for each unbounded end $e\in\Gamma_\infty^1$, $h(e)\in\mathcal{L}_e$ ?
\end{prob}

In other words, using previous notations, solving this enumerative problem means finding the preimages of a generic point for the following composed evaluation map
$$\mathrm{ev}^{(L_e)}:\modpar\rightarrow\prod_e Q_e.$$
The enumerative problem is referred as $\mathcal{P}(L_e)$, and is related to an analogous complex enumerative problem though the use of a correspondence theorem. The count with the complex multiplicity $m^{\CC,(L_e)}_\Gamma$ is already known to lead to an invariant. The result is recalled here is proven in the appendix.

\begin{theom}[\ref{proposition invariance tropical count with complex multiplicities}]
For a generic choice of $(\mathcal{L}_e)$, the count $N_\Delta^\CC(\mathcal{L}_e)=\sum_{\Gamma:h(e)\in\mathcal{L}_e}m_\Gamma^{\CC,(L_e)}$ does not depend on the choice of $(\mathcal{L}_e)$. It only depends on the choice of $\Delta$ and $(L_e)$.
\end{theom}

Furthermore, the result of this count is equal to the complex invariant provided by the analogous complex problem. However, the involved complex multiplicity $m^{\CC,(L_e)}$ does not in general split into a product over the vertices of the curve. Nevertheless, we prove that for a suitable choice of $K$ and up to a change of sign, the refined multiplicity $B_\Gamma^K$, which is a product over the vertices, does lead to a invariant. Meanwhile, let us describe the combinatorial types that have complex multiplicity $0$.

\begin{lem}
\label{lemma zero mult}
A parametrized tropical curve $h:\Gamma\rightarrow N_\RR$ with complex multiplicity $0$ falls into at least one of the following categories:
\begin{enumerate}[label=(\roman*)]
\item The curve has a contracted bounded edge, \textit{i.e.} it is reducible.
\item The curve has a flat vertex, \textit{i.e.} $\pi_V=0$ for some $V$.
\item For some bounded edge $\gamma\in\Gamma^1_b$, the ranks of the constraints are wrongly dispatched: let $A\sqcup B$ be the partition of $\Gamma_\infty^1$ induced by $\gamma$, then
$$\sum_{e\in A} \mathrm{cork}\ L_e>r+|A|-2 \text{ or }\sum_{e\in B} \mathrm{cork}\ L_e>r+|B|-2.$$
\item The curve has no flat vertex or contracted edge, the ranks of constraints are rightly dispatched but non-generic.
\end{enumerate}
\end{lem}

\begin{proof}
In the first two cases, the multiplicity is zero since the evaluation map is non-injective, so neither is the composed evaluation map. Let us focus on the third case and assume there is a bounded edge $\gamma\in\Gamma^1_b$ such that the constraints are wrongly dispatched. Out of symmetry, let assume that $\sum_{e\in A} \mathrm{cork}\ L_e>r+|A|-2$ and the sink used to compute the multiplicity is in the $B$ part of the curve. The recursive formula $\rho=\iota_{n+n'}(\rho_1\wedge\rho_2)$ used to compute the complex multiplicity shows that the polyvector $\rho_\gamma$ obtained applying the algorithm is the interior product with a polyvector belonging to $\Lambda^{\sum_A\mathrm{cork}L_e -|A|+2}M$. By assumption, this space is $0$ since the weight is strictly bigger than $r$, and so is the multiplicity.

By disjunction, the last case contains indeed the remaining cases. Using the Pl\"ucker embedding, the space of corank $l$ sublattices of a given lattice $N$ is $\Lambda^l M$. As $\mathrm{Hom}(N/\gen{n_e},\ZZ)\simeq \gen{n_e}^\perp\subset M$, the complex multiplicity can be seen as a multilinear form on the space $\prod_e \Lambda^{\mathrm{cork}L_e}\gen{n_e}^\perp$ of possible slopes for the lattices $L_e$, once their dimensions are fixed. The assumption ensures that the multilinear form itself is non-zero, but it might still take the zero value on some specific choice of slopes. This is avoided if the slopes are chosen generically.
\end{proof}

\subsection{Main results}

For a choice of slope constraints $(L_e)$, let $\omega\in\Lambda^2 M$ be a $2$-form satisfying the following property: for every top-dimensional combinatorial type $\Gamma$,
$$m^{\CC,(L_e)}_\Gamma=0\Rightarrow m^{\CC,\omega,e_0}_\Gamma=0.$$
In particular, if for some combinatorial type $\Gamma$ which is not of type $(i)$ or $(ii)$ one has $m^{\CC,(L_e)}_\Gamma=0$, then $K_\omega\neq\{0\}$. Such a $\omega$ may not be unique and in some cases might even be $0$. However, the choice of $\omega$ does not really matter, only the space $K_\omega$ matters.

\begin{rem}
It is an interesting question to decide whether there exists a non-trivial $\omega$ satisfying these requirements. These conditions need only be checked on combinatorial types without contracted edges nor flat vertex. Nevertheless, to prove the interest of the results, we provide a wide class of examples which admit such an $\omega$.
\end{rem}

\begin{expl}
\label{example wide class}
Choose the boundary constraints on every unbounded end to be at least codimension $1$ and having generic slope, and let $\omega$ be a $2$-form such that $K_\omega=\{0\}$, meaning that the only combinatorial types with complex multiplicity relative to $\omega$ equal to $0$ are the curves with a flat vertex or a contracted edge. Then, we prove that the complex multiplicity relative to $(L_e)$ is never $0$ unless there is a flat vertex or contracted edge.

According to the classification of Lemma \ref{lemma zero mult}, we can only be in the case $(iii)$. Thus, we show that the dimensions are rightly dispatched. Let $\gamma\in\Gamma^1_b$ be an edge and $A\sqcup B=\Gamma_\infty^1$ the associated partition. Then by assumption,
$$\sum_{e\in A}\mathrm{cork}L_e\geqslant|A| \text{ and }\sum_{e\in B}\mathrm{cork}L_e\geqslant|B|.$$
Moreover, one has $\sum_{e}\mathrm{cork}L_e=|\Delta|+r-3$. Thus,
$$\left( \sum_{e\in A}\mathrm{cork}L_e -|A|\right)
+ \left( \sum_{e\in B}\mathrm{cork}L_e -|B|\right)=r-3,$$
and as each term on the left-hand side is positive, neither can be bigger than $r-2$. Hence, constraints are rightly dispatched
\end{expl}

We now get to the main result. We use the refined multiplicity $B_\Gamma^{K_\omega}$, where the bivectors $\pi_V$ are chosen such that $\omega(\pi_V)>0$. If $\omega(\pi_V)=0$, then $B_\Gamma^{K_\omega}=0$ and the sign does not matter.

\begin{theo}\label{theorem invariance general}
There exists a collection of signs $\varepsilon_\Gamma$ indexed by the combinatorial types $\mathcal{C}_\Delta^\mathrm{top}$ such that the following refined count
$$\mathcal{B}_\Delta\left( (\mathcal{L}_e)\right)=\sum_{\Gamma:h(e)\in\mathcal{L}_e} \varepsilon_\Gamma B^{K_\omega}_\Gamma \in\ZZ[\Lambda^2 N/K_\omega],$$ does not depend on the choice of the constraints $\mathcal{L}_e$ as long as it is generic. It only depends on the choice of $L_e$.
\end{theo}

\begin{rem}
It might happen that for some combinatorial type $m_\Gamma^{\CC,(L_e)}\neq 0$ and still $B_\Gamma^{K_\omega}=0$, but not the other way around.
\end{rem}

The construction of the collection of signs is given in the proofs. They can be obtained as follows. Let $C$ be a top-dimensional combinatorial type in $\mathcal{C}_\Delta^\mathrm{top}$ with non-zero complex multiplicity $m_C^{\CC,(L_e)}$. We have two evaluation maps restricted to the cone $C$:
$$\begin{array}{cccl}
\mathrm{ev}_C^{(L_e)}: & \RR_+^{|\Delta|-3}\times N_\RR & \longrightarrow & \prod_e Q_e,\\
\mathrm{ev}_C^{\omega,e_0}: & \RR_+^{|\Delta|-3}\times N_\RR & \longrightarrow & N_\RR/\gen{\delta}\times\prod_{e\neq e_0} \RR.\\
\end{array}$$
Let us assume that both maps are invertible, otherwise the definition of the sign does not matter since $B_C^{K_\omega}=0$. Let us choose once and for all an orientation of the codomains, \textit{i.e.} not depending on the chosen combinatorial type $\Gamma$. For any orientation on $C=\RR_+^{|\Delta|-3}\times N_\RR$, if the maps both preserve the orientation or both reverse it, the sign is $+$. If one of them preserves the orientation and the other reverses it, the sign is $-$. The signs $\varepsilon_C$ are thus defined up to a global sign depending on the choice of the identification between the codomains of both evaluation maps.

In other words, if one identifies the codomains with $\RR^{|\Delta|+r-3}$ and one chooses an orientation on it, the sign $\varepsilon_C$ compares the orientations on $C$ which are induced by both composed evaluation maps.

\subsection{Proofs of invariance}

We give two different proofs of the result. Both proofs use a trick that might be of interest in other situations as well: without giving to much context, given two tropical enumerative problems $\mathcal{P}$ and $\mathcal{P}'$ for tropical curves having the same walls, assuming some additional condition on the combinatorial types not providing any solution, it is possible to use a multiplicity $m$ leading to an invariant in $\mathcal{P}$ to get an invariant in $\mathcal{P}'$, eventually twisting the multiplicity $m$ with a sign. In our case, we use the invariance with respect to the refined multiplicity $B^{K_\omega}$ in the $\omega$-problem $\mathcal{P}(\omega,e_0)$ to get an invariance in the problem $\mathcal{P}(L_e)$.

\begin{proof}[First proof of Theorem \ref{theorem invariance general}]
The first way of proving the existence of the signs is choosing the sign of one combinatorial type and spreading the signs so that one has local invariance around the walls, which are here the curves with a quadrivalent vertex. Then, one just has to check that the signs are well-defined.

Let $\mathcal{W}_\Delta$ be the set of walls, \textit{i.e.} combinatorial types of curves with a quadrivalent vertex. Consider the graph on $\mathcal{C}_\Delta^\mathrm{top}$ where the two combinatorial types $C$ and $C'$ are adjacent if they are adjacent to a common wall $W\in\mathcal{W}_\Delta$. We consider the subgraph $\mathcal{G}_\mathrm{reg}\subset\mathcal{C}_\Delta^\mathrm{top}$ only containing the combinatorial types with non-zero multiplicity $B_C^{K_\omega}$, and thus non-zero complex multiplicity $m_C^{\CC,(L_e)}$, along with the edges between them. This graph might not be connected anymore. For a given connected component, let $C_0$ be a fixed combinatorial type, thus having non-zero complex multiplicity $m_{C_0}^{\CC,(L_e)}$. Let $\varepsilon_{C_0}=1$. For a given wall $W\in\mathcal{W}_\Delta$, we write the local invariance relation for $B^{K_\omega}$ around $W$ in the following way:
$$\eta_{C_1}^W B^{K_\omega}_{C_1}
+\eta_{C_2}^W B^{K_\omega}_{C_2}
+\eta_{C_3}^W B^{K_\omega}_{C_3}
=0\in\ZZ[\Lambda^2 N/K_\omega],$$
where $\eta_C^W$ are equal to $\pm 1$, and depend on the repartition of the sides of the wall in the enumerative problem $\mathcal{P}(\omega,e_0)$. The triple $(\eta^W_{C_1},\eta^W_{C_2},\eta^W_{C_3})$ is defined up to global change of sign. If for a given combinatorial type $C$ one has $B_C^{K_\omega}=0$, the signs $\eta_C^\bullet$ are not uniquely defined but their definition is unimportant. Now, for a combinatorial type $C\in\mathcal{G}_\mathrm{reg}\subset\mathcal{C}_\Delta^\mathrm{top}$, let
$$C_0\stackrel{W_0}{\longrightarrow} C_1 \stackrel{W_1}{\longrightarrow} \cdots \stackrel{W_{n-1}}{\longrightarrow} C_n=C$$
be a path going from $C_0$ to $C$ in $\mathcal{G}_\mathrm{reg}$, where $W_i\in\mathcal{W}_\Delta$ is a wall adjacent to both combinatorial types $C_i$ and $C_{i+1}$. Now, we use the following rule to define the signs $\varepsilon_{C_i}$: let $C'_i\in\mathcal{C}_\Delta^\mathrm{top}$ be the last combinatorial type adjacent to $W_i$, so that the adjacent combinatorial types are $C_i$, $C_{i+1}$ and $C'_i$. We are in one of the following situations:
\begin{itemize}
\item The combinatorial types $C_i$ and $C_{i+1}$ are on the same side of the wall, so that the invariance relation for the complex multiplicity $m^{\CC,(L_e)}$ writes itself
$$m_{C_i}^{\CC,(L_e)}+m_{C_{i+1}}^{\CC,(L_e)}=m_{C'_i}^{\CC,(L_e)}.$$
We write the invariance relation for $B^{K_\omega}$ mimicking the previous relation, and multiply it so that the coefficient of $m^{K_\omega}_{C_i}$ is $\varepsilon_{C_i}$, leading to
$$\varepsilon_{C_i}m_{C_i}^{K_\omega}
+ \varepsilon_{C_i}\frac{\eta^{W_i}_{C_{i+1}}}{\eta^{W_i}_{C_i}}m_{C_{i+1}}^{K_\omega}
=-\varepsilon_{C_i}\frac{\eta^{W_i}_{C'_i}}{\eta^{W_i}_{C_i}}m_{C'_i}^{K_\omega}.$$
Thus, we choose
$$\varepsilon_{C_{i+1}}=\varepsilon_{C_i}\frac{\eta^{W_i}_{C_{i+1}}}{\eta^{W_i}_{C_i}}.$$
\item If $C_i$ and $C_{i+1}$ are on different sides of the walls, the invariance for the complex multiplicity writes itself
$$m_{C_i}^{\CC,(L_e)}+m_{C'_i}^{\CC,(L_e)}=m_{C_{i+1}}^{\CC,(L_e)} \text{ or }m_{C_i}^{\CC,(L_e)}=m_{C_{i+1}}^{\CC,(L_e)}+m_{C'_i}^{\CC,(L_e)}.$$
We then proceed similarly and choose
$$\varepsilon_{C_{i+1}}=-\varepsilon_{C_i}\frac{\eta^{W_i}_{C_{i+1}}}{\eta^{W_i}_{C_i}}.$$
\item If $m^{\CC,(L_e)}_{C'_i}=0$, $C_i$ and $C_{i+1}$ are on different sides of the wall, and the combinatorial type $C'_i$ never provides any solution to the enumerative problem for a generic value. As by assumption $B^{K_\omega}_{C'_i}$ is also $0$, the previous case still works.
\end{itemize}
For each connected component of $\mathcal{G}_\mathrm{reg}$, the choice of a sign on one combinatorial type $C_0$ thus propagates to define a sign for every other combinatorial type in the connected component. The rule to propagate the sign from one side of the wall to another might be reformulated in the following closed way. Fix an orientation on the codomain of the composed evaluation map $\mathrm{ev}^{(L_e)}$. Let $\ori(C)$ be the orientation of $C$ induced  $\mathrm{ev}^{(L_e)}_C$. For a combinatorial type on which the composed evaluation map is not invertible, the orientation is chosen at random and does not matter. Then, two combinatorial types $C_i$ and $C_{i+1}$ stand on the same side of the wall $W_i$ if and only if the induced orientations $\partial_{W_i}\ori(C_i)$ and $\partial_{W_i}\ori(C_{i+1})$ on $W_i$ coincide. Thus, we can use the closed formula
$$\varepsilon_{C_{i+1}}=\varepsilon_{C_i}\frac{\partial_{W_i}\ori(C_{i+1})}{\partial_{W_i}\ori(C_i)}\frac{\eta^{W_i}_{C_{i+1}}}{\eta^{W_i}_{C_i}}.$$

We now check that the signs are well-defined, meaning that the definition of $\varepsilon_C$ does not depend on the chosen path between $C_0$ and $C$. Notice that the rule of propagation is reversible: the sign of $C_i$ determines the sign of $C_{i+1}$ in the same way that the sign of $C_{i+1}$ determines the sign $C_i$. Thus, we only have take a loop from $C_0$ to itself, and prove that the sign obtained by the propagation rule is equal to the original sign $\varepsilon_{C_0}=+1$. Let
$$C_0\stackrel{W_0}{\longrightarrow} C_1 \stackrel{W_1}{\longrightarrow} \cdots \stackrel{W_{n-1}}{\longrightarrow} C_n\stackrel{W_{n}}{\longrightarrow}C_0$$
be such a loop inside a connected component of $\mathcal{G}_\mathrm{reg}$. Then, the iteration of the rule of signs gives the following sign on $C_0$:
$$\prod_{i=0}^n \frac{\partial_{W_i}\ori(C_{i+1})}{\partial_{W_i}\ori(C_i)}
\frac{\eta^{W_i}_{C_{i+1}}}{\eta^{W_i}_{C_i}},$$
and needs to be equal to $1$. This can be written as a condition on the signs $\eta$:
$$\prod_{i=0}^n \frac{\eta^{W_i}_{C_{i+1}}}{\eta^{W_i}_{C_i}} = \prod_{i=0}^n \frac{\partial_{W_i}\ori(C_{i+1})}{\partial_{W_i}\ori(C_i)}.$$
Splitting the right-hand side product and reindexing its numerator, we get to
$$\prod_{i=0}^n \frac{\eta^{W_i}_{C_{i+1}}}{\eta^{W_i}_{C_i}} = \prod_{i=0}^n \frac{\partial_{W_{i-1}}\ori(C_i)}{\partial_{W_i}\ori(C_i)}.$$
Notice that this condition does not depend on the up to local sign choice of the $\eta$, and does neither depend on the $\eta_C^\bullet$ for $C$ with $B_C^{K_\omega}=0$ since the loop stays inside a component of $\mathcal{G}_\mathrm{reg}$. As each factor of the right-hand side is unchanged if the orientation of a combinatorial type $C_i$ is reversed, then the condition does not depend on the orientations of each $C$. Therefore, the rule to propagate signs depends on the orientations $\ori(C)$ chosen on the combinatorial types, but the condition to ensure that these signs are uniquely defined only depends on the $\eta_C^W$, not on the chosen orientations. Thus, we can choose any orientation to check the formula, for instance, the orientation induced by the composed evaluation map $\mathrm{ev}^{\omega,e_0}$, related to $\mathcal{P}(\omega,e_0)$.

As the refined multiplicity $B^{K_\omega}$ is known to give an invariant in the problem $\mathcal{P}(\omega,e_0)$, the local invariance around a wall $W_i$ gives by definition $\eta_C^W=\frac{\partial_W\ori_\omega(C)}{\ori(W)}$, for a fixed orientation $\ori(W)$ of $W$. Thus,
$$\frac{\eta_{C_{i+1}}^{W_i}}{\eta_{C_i}^{W_i}}=\frac{\partial_{W_i}\ori_\omega(C_{i+1})}{\partial_{W_i}\ori_\omega(C_i)}.$$
It implies that the relation is verified, which concludes the proof since by construction the local invariance is satisfied at each wall.
\end{proof}

\begin{rem}
We could have noticed from the start that the $\eta_C^W$ are given by the formula $\frac{\partial_W\ori_\omega(C)}{\ori(W)}$ and conclude that the necessary relation to guarantee that signs were well-defined is satisfied. However, this wouldn't emphasize the fact the condition solely depends on the signs in front on the multiplicities in the dependence relations given by the walls.
\end{rem}

This first proof gives the existence of the signs but fails to provide a useful definition, which is given in the following second proof, whose idea consists in interpreting the multiplicities leading to an invariant as cycles on the moduli space $\modpar$.

\begin{proof}[Second proof of Theorem \ref{theorem invariance general}]
Consider the orientation $\ori(C)$ of $C\in\mathcal{C}_\Delta^\mathrm{top}$ induced by the composed evaluation map $\mathrm{ev}^{(L_e)}$ and a fixed orientation of its codomain. If the restriction to a given combinatorial type is not invertible, the chosen orientation does not matter. We then consider the following top dimensional chain
$$\Xi=\sum_{C\in\mathcal{C}_\Delta^\mathrm{top}} m_C^{(L_e)} (C,\ori(C)) \in \ZZ^{\mathcal{C}_\Delta^\mathrm{top} }=C_\mathrm{top}(\modpar).$$
Technically, this is a chain in the star complex of the fan $\modpar$, obtained as the quotient of $\modpar\backslash\{0\}$ by $\RR_+^*$. The fan $\modpar$ is the cone over its star complex, which is compact a CW-complex. Nevertheless, we write $\modpar$ to avoid further notations which are not essential to the proof, denoting the cells of the star complex by their associated cone in $\modpar$. Choosing an arbitrary orientation $\ori(W)$ of every wall, the boundary of $\Xi$ is
$$\partial \Xi=\sum_{W\in\mathcal{W}_\Delta} \left( \eta^W_{C_1} m_{C_1}^{(L_e)} + \eta^W_{C_2} m_{C_2}^{(L_e)} + \eta^W_{C_3} m_{C_3}^{(L_e)} \right) (W,\ori(W)),$$
where $C_1,C_2$ and $C_3$ are the sides adjacent to the wall $W$, and the sign $\eta^W_{C_i}=\frac{\partial_W\ori(C)}{\ori(W)}$ is given by the boundary map. Then, the local invariance relation  needed to show that multiplicities $m^{(L_e)}$ lead to an invariant for the enumerative problem $\mathcal{P}(L_e)$ is equivalent to the fact that $\Xi$ is a cycle: $\partial\Xi=0$. Conversely, any $m=(m_C)_C$ such that
$$\partial\left( \sum_{C\in\mathcal{C}_\Delta^\mathrm{top}} m_C (C,\ori(C)) \right)=0 \text{ and }m_C^{\CC,(L_e)}=0\Rightarrow m_C=0,$$
provides an invariant for the enumerative problem $\mathcal{P}(L_e)$. The first relation ensures the local invariance statement while the second condition ensures that combinatorial types which do not provide any solution do not interfere in the local invariance.

Now, let $\ori_\omega(C)$ be the orientation of $C$ induced by the composed evaluation map $\mathrm{ev}^{\omega,e_0}$. Then, as the refined multiplicity $B^{K_\omega}$ provides an invariance in the enumerative problem $\mathcal{P}(\omega,e_0)$, one has
$$\partial\left( \sum_{C\in\mathcal{C}_\Delta^\mathrm{top}} B_C^{K_\omega} (C,\ori_\omega(C)) \right)=0.$$
Thus, twisting the signs, one gets
$$\partial\left( \sum_{C\in\mathcal{C}_\Delta^\mathrm{top}} \frac{\ori_\omega(C)}{\ori(C)} B_C^{K_\omega} (C,\ori(C)) \right)=0.$$
As by assumption $m_C^{\CC,(L_e)}=0\Rightarrow B_C^{K_\omega}=0$, the choice $\varepsilon_C=\frac{\ori_\omega(C)}{\ori(C)}$ gives an invariant in the enumerative problem $\mathcal{P}(L_e)$ and concludes the proof.
\end{proof}

\begin{rem}
The second proof allows one to view the multiplicities leading to an invariant as some homology class in $H_\mathrm{top}\left( \mathrm{Star}(\modpar),R\right)$ with values in some ring $R$, $\ZZ$ or $\ZZ[\Lambda^2 N]$ for instance.
\end{rem}

\begin{rem}
The second proof gives an explicit description of signs. However, the first proof  might provide several possible signs: one choice of signs for each connected component of the graph $\mathcal{G}_\mathrm{reg}$. Recall that $\mathcal{G}_\mathrm{reg}$ is the graph of combinatorial types whose multiplicity is non-zero. It would be an interesting question to know whether this graph is always connected or not. If the graph were not connected, the sum of multiplicities of curves solution to an enumerative problem and whose combinatorial type belongs to a given component of $\mathcal{G}_\mathrm{reg}$ would lead to refined invariants of a new kind.
\end{rem}

\section{First properties and possible enhancements of refined invariants}
\label{section first properties and possible}

We now focus ourselves on some properties of these new refined invariants, including their relation to the complex multiplicities. Then, we provide some directions to enlarge and improve, or not, the family of refined invariants.

\subsection{Relation to the complex multiplicities}

We here prove a relation between the refined and the complex multiplicities in the case of the $\omega$-problem $\mathcal{P}(\omega,e_0)$. In the general setting, to some extent with prove the non-existence of such a relation. The spirit of this section is to find a way to recover $m^{\CC}_\Gamma$ from $B_\Gamma^K$.\\

Let us first consider the planar setting, \textit{i.e.} $r=2$, and let $\mathrm{det}$ be a generator of $\Lambda^2 M$. Recall that the refined multiplicity of Block-G\"ottsche is defined as follows: for a trivalent parametrized tropical curve $h:\Gamma\rightarrow N_\RR$,
$$m^{BG}_\Gamma=\prod_{V\in\Gamma^0}\frac{q^{\mathrm{det}(\pi_V)/2}-q^{-\mathrm{det}(\pi_V)/2}}{q^{1/2}-q^{-1/2}}\in\ZZ[q^{\pm 1/2}],$$
where $\pi_V$ is chosen so that $\mathrm{det}(\pi_V)>0$. This refined multiplicity is related to our choice of refined multiplicity through the relation
$$(q^{1/2}-q^{-1/2})^{|\Delta|-2}m_\Gamma^{BG}=\varphi_\mathrm{det}\left( B^{0}_\Gamma\right),$$
where $\varphi_\mathrm{det}:\ZZ[\Lambda^2 N]\rightarrow \ZZ[q^{\pm 1/2}]$ is the morphisms that evaluates $\mathrm{det}/2$ on the exponents. As $m^{BG}_\Gamma\xrightarrow{q\rightarrow 1}\prod_V\mathrm{det}(\pi_V)=m^{\CC}_\Gamma$, the refined multiplicity determines the complex multiplicity in the planar case, and so the refined invariant $\mathcal{B}_\Delta^{\omega,e_0}$ determines the complex invariant $N_\Delta^{\omega,e_0}$.\\

Getting back to the general case, this limit relation extends to any dimension, and thus proves that the refined multiplicity determines the complex multiplicity in the $\omega$-problem $\mathcal{P}(\omega,e_0)$ for any $\omega$, which is stated in the following proposition.

\begin{prop}
For any choice of $\omega$, one has $(q-q^{-1})^{2-|\Delta|}\varphi_\omega(B_\Gamma^{K_\omega})\xrightarrow{q\rightarrow 1}m^{\CC,\omega,e_0}_\Gamma$, and thus
$$(q-q^{-1})^{2-|\Delta|}\varphi_\omega(\mathcal{B}_\Delta^{\omega,e_0})\xrightarrow{q\rightarrow 1}N_\Delta^{\omega,e_0}.$$
\end{prop}

In the general case of the enumerative problem $\mathcal{P}(L_e)$, the question remains to know whether the complex multiplicity can be recovered using the refined multiplicity and some clever function involving the constraints. Sadly, the complex multiplicity cannot be determined solely with the refined multiplicity.

\begin{rem}
It seems possible to find tropical curves satisfying constraints of the same slope, and with the same refined multiplicity but different complex multiplicities, preventing from finding a simple relation between the refined multiplicity and the complex multiplicity.
\end{rem}

\subsection{Regularity of the invariants}

For a fixed degree $\Delta$, the invariants $\mathcal{B}_\Delta^{\omega,e_0}$ and $\mathcal{B}_\Delta^{K}(L_e)$ form a large family of invariants depending on $\omega$ for the first family, and $(L_e)$ for the second The $2$-form $\omega$ varies in the space $\Lambda^2 M\otimes\RR$, while $(L_e)$ varies in $\prod_e \Lambda^{l_e}\gen{n_e}^\perp$, where $l_e=\mathrm{cork}L_e$. In this section we study the dependence of these invariants in their respective parameters.\\

The case of the invariants $\mathcal{B}_\Delta^{\omega,e_0}$ is dealt with in the following theorem. First, let us introduce the fan structure $\Omega_\Delta$ on $\Lambda^2M_\RR$, which is induced by the linear forms $\omega\mapsto\omega(\pi_V)$, for every $V\in\mathcal{V}_\Delta$. Notice that the function $\omega\mapsto K_\omega=\gen{\pi_V:\omega(\pi_V)=0}_{V\in\mathcal{V}_\Delta}$ is constant on the cones of the subdivision, and thus denoted by $K_\sigma$ for a given cone $\sigma$ of $\Omega_\Delta$. Furthermore, if $\tau\prec\sigma$ are cones, then $K_\sigma\subset K_\tau$. The top-dimensional cones are those for which $K_\sigma=\{0\}$. Remember that $\mathcal{B}_\Delta^{\omega,e_0}$ is defined whenever $\Delta\cap\{n:\iota_n\omega=0\}=\emptyset$. The set of $\omega$ not satisfying this assumption is some subfan of $\Omega_\Delta$, where one can set $\mathcal{B}_\Delta^{\omega,e_0}=0$.

\begin{theo}
\label{theorem continuity moment}
\begin{itemize}
\item The function $\omega\mapsto\mathcal{B}_\Delta^{\omega,e_0}$ is constant on the interior of the cones of $\Omega_\Delta$ where it is defined.
\item If $\tau\prec\sigma$ and $\tau$ are cones of $\Omega_\Delta$, with values $\mathcal{B}(\sigma)$ and $\mathcal{B}(\tau)$, then $\mathcal{B}(\tau)$ is the image of $\mathcal{B}(\sigma)$ under the projection $\ZZ[\Lambda^2N/K_\sigma]\rightarrow\ZZ[\Lambda^2N/K_\tau]$.
\end{itemize}
\end{theo}

\begin{proof}
Let $\omega_0$ be a fixed $2$-form where $\mathcal{B}_\Delta$ is defined. Let $C\in\mathcal{C}_\Delta^\mathrm{top}$ be a combinatorial type with non-zero multiplicity $m_C^{\CC,\omega_0,e_0}$ and $\Gamma_0$ be a curve belonging to the interior of $C$. We consider $\mathrm{ev}^{\omega,e_0}$ as a function of both $\omega$ and $\Gamma\in C$. This way, the partial differential $\frac{\partial \mathrm{ev}^{\omega,e_0}}{\partial \Gamma}$ at $(\omega_0,\Gamma_0)$ is the usual differential of $\mathrm{ev}^{\omega_0,e_0}$, which is itself since it is linear on $C$. Therefore, the differential is invertible for its determinant is equal to the complex multiplicity $m_{\Gamma_0}^{\CC,\omega_0,e_0}$, non-zero by assumption. Thus, the implicit function theorem ensures that at the neighborhood $(\omega_0,\Gamma_0)$, the equation $\mathrm{ev}^{\omega,e_0}(\Gamma)=\mathrm{ev}^{\omega_0,e_0}(\Gamma_0)$ in $(\omega,\Gamma)$ solves itself with $\Gamma$ being a function $\Gamma(\omega)$, with $\Gamma(\omega_0)=\Gamma_0$.

Let us use the previous fact to prove the theorem. For a generic value $(\mu,D)$, at $\omega_0$, one has
$$\mathcal{B}_\Delta^{\omega_0,e_0}=\sum_{\mathrm{ev}^{\omega_0}(\Gamma)=(\mu,D)}B_\Gamma^{K_{\omega_0}}.$$
For the combinatorial types that contribute a solution in the above sum, the previous fact using implicit functions theorem shows that they still provide a solution when $\omega$ moves in the neighborhood of $\omega_0$. For the same reasons, combinatorial types where the composed evaluation map is invertible  but which do not provide a solution continue not to provide a solution. The remaining combinatorial types never provide a solution for $\mathcal{P}(\omega_0,e_0)$, but might start to do so for some $\mathcal{P}(\omega,e_0)$. However, their multiplicity $B_\Gamma^{K_{\omega_0}}$ is $0$. Moreover, the function
$$\omega\longmapsto B_\Gamma^{K_{\omega}}=\prod_V(q^{\pi_V}-q^{-\pi_V})\in\ZZ[\Lambda^2N/K_{\omega_0}],$$
is constant in a neighborhood of $\omega_0$: if the value is $0$, it stays zero since some $\pi_V$ belongs to $K_{\omega_0}$, otherwise the conditions $\omega(\pi_V)>0$ remains valid for $\omega$ in a neighborhood of $\omega_0$. In particular, for $\omega$ belonging to the same cone of $\Omega_\Delta$, since then one has $K_\omega=K_{\omega_0}$, the multiplicity stays the same, and $\omega\mapsto\mathcal{B}_\Delta^{\omega,e_0}$ is locally constant for $\omega$ in the neighborhood of $\omega_0$ and in the same cone. The first statement follows since cones are connected.

If $\omega$ is still in the neighborhood of $\omega_0$ but does not belong to the same cone of $\Omega_\Delta$, then the same observation applies. And $\mathcal{B}_\Delta^{\omega,e_0}=\mathcal{B}_\Delta^{\omega_0,e_0}\in \ZZ[\Lambda^2N/K_{\omega_0}]$ which means that $\mathcal{B}_\Delta^{\omega_0,e_0}$ is the image of $\mathcal{B}_\Delta^{\omega,e_0}$ under the map that projects the exponents $\ZZ[\Lambda^2N/K_{\omega}]\rightarrow\ZZ[\Lambda^2N/K_{\omega_0}]$, thus proving the second statement.
\end{proof}

\begin{rem}
Clearly, the $\omega$-problem $\mathcal{P}(\omega,e_0)$ depends on $\omega$ only up to multiplication by a scalar. The fan $\Omega_\Delta$ induces a subdivision on the projective space $\mathbb{P}(\Lambda^2 M_\RR)$, which is a hyperplane arrangement. Then, $[\omega]\mapsto\mathcal{B}_\Delta^{\omega,e_0}$ is a function constant on the cells of the subdivision and satisfying the compatibility condition when going from a cell to one of its faces. On the cells where $\mathcal{B}_\Delta^{\omega,e_0}$ is not defined, the map can be extended by $0$. This choice is natural since for such $2$-forms, the composed evaluation map is never surjective, and there is an invariant which has value $0$.
\end{rem}

For the general enumerative problem $\mathcal{P}(L_e)$, the same reasoning applies. Let $(L_e^0)$ be a specific set of slopes constraints. Let $l_e=\mathrm{cork}L_e^0$, so that the tuple $(L_e)$ varies in $\mathbb{L}=\prod_e\Lambda^{l_e}\gen{n_e}^\perp$.

\begin{theo}
\label{theorem continuity general}
Let $(L_e^0)\in\mathbb{L}$ be a specific set of slopes constraints. Let $\omega$ be a $2$-form satisfying
$$m_C^{\CC,(L^0_e)}=0\Rightarrow m_C^{\omega,e_0}=0,$$
leading to an invariant $\mathcal{B}_\Delta^{K_\omega}(L_e^0)\in\ZZ[\Lambda^2 N/K_\omega]$. Then $\omega$ also satisfies the condition for $(L_e)$ in a neighborhood of $(L_e^0)$, and the function $(L_e)\mapsto\mathcal{B}_\Delta^{K_\omega}(L_e)$ is constant in a neighborhood of $(L_e^0)$.
\end{theo}

\begin{proof}
The condition
$$m_C^{\CC,(L^0_e)}=0\Rightarrow m_C^{\omega,e_0}=0,$$
can be restated as
$$m_C^{\omega,e_0}\neq 0 \Rightarrow m_C^{\CC,(L^0_e)}\neq 0,$$
Furthermore, for each combinatorial type $C\in\mathcal{C}_\Delta^\mathrm{top}$, the complex multiplicity seen as a function $(L_e)\mapsto m_C^{\CC,(L_e)}$ is continuous on $\mathbb{L}$. Thus, in a neighborhood of $(L_e^0)$, one still has $m_C^{\CC,(L_e)}\neq 0$. Thus, the $2$-form $\omega$ also satisfies the condition for $(L_e)$ and thus provides an invariant in $\ZZ[\Lambda^2 N/K_\omega]$. The implicit functions theorem provides the local invariance in the same way as for the $\omega$-problem case: combinatorial types where the complex multiplicity is non-zero provide or not a solution for a specific value of the constraints, and we apply the implicit functions theorem, and combinatorial type where the complex multiplicity cancels might provide a solution for $(L_e)$ in a neighborhood of $(L_e^0)$ but their multiplicity $B_C^{K_\omega}$ is $0$ by assumption.
\end{proof}

\subsection{Tropical cycles constraints}

We now try to enlarge the scope of the refined invariance theorems by generalizing the family of enumerative problems $\mathcal{P}(L_e)$ involved in Theorem \ref{theorem invariance general}. The first direction in which one could look is to replace the constraints, which are for now affine subspaces in $N_\RR/\gen{n_e}$, by tropical cycles of the same dimension inside the same space. Let $(\Xi_e)$ be a family of tropical cycles $\Xi_e\subset N_\RR/\gen{n_e}$ such that $\sum_e\mathrm{codim}\Xi_e=|\Delta|+r-3$.

\begin{prob}
The problem referred as $\mathcal{P}(\Xi_e)$ is counting rational tropical curves of degree $\Delta$ which intersect $\Xi_e$ for every $e\in\Gamma_\infty^1$.
\end{prob}

The dimension count ensures that if the cycles $\Xi_e$ are chosen generically, there is a finite number of solutions. In this situation, general theory still provides a complex multiplicity which leads to an invariant. However, the proof of Theorem \ref{theorem invariance general} for the refined invariance fails in this context for the following reasons.

\begin{itemize}
\item First, in the case of affine constraints, each tropical curve $h:\Gamma\rightarrow N_\RR$ determines uniquely the constraints that it satisfies: $\mathcal{L}_e$ is the affine space with slope $L_e$ passing through $h(e)\in N_\RR/\gen{n_e}$. This can be reprased as the existence of the composed evaluation map $\modpar\rightarrow \prod_e Q_e^\RR$, where $Q_e^\RR$ is a moduli space for the affine subspaces $\mathcal{L}_e$ of slope $L_e$. The proof of invariance then consists in showing that this map as a well-defined degree for various multiplicities. This is not the case for tropical cycle constraints $(\Xi_e)$ since they vary in a wider moduli space: there is no composed evaluation map anymore. Thus, there is no suitable definitions of the "walls" inside $\modpar$ that could lead to an analogous proof.
\item Then, if one tries instead to mimic the invariance proof for the complex multiplicity, whether it uses tropical intersection theory or not, we have to check invariance around one new kind of wall: where the intersection point between one tropical curve $\Gamma$ solution to the enumerative problem and one of the constraints $\Xi_e$ belongs to a facet of $\Xi_e$. This intersection is not generic among the choices of $\Xi_e$, and the invariance for the complex multiplicity is provided by the balancing condition of $\Xi_e$, which would express as $\rho_1+\rho_2+\rho_3=0$ for the Pl\"ucker vectors of the adjacent facets of $\Xi_e$. However, at this wall, two solutions on one side of the wall become one on the other side, but the combinatorial types of the solutions are the same. This prevents any local invariance using a multiplicity solely depending on the combinatorial type of the curve, as is $B_\Gamma^K$.
\item Despite the impossibility of the multiplicity $B_\Gamma^K$ to provide a local refined invariant, it might exists a clever way to find another refined multiplicity leading to an invariant. However, such a multiplicity would not have the close relationship that $B_\Gamma^0$ shares with the count of real curves tropicalizing to $\Gamma$ through the correspondence theorem.
\end{itemize}

Nevertheless, we can prove refined invariance in the following particular case. Let $\Xi_{e_0}\subset N_\RR/\gen{e_0}$ be a tropical cycle of dimension $1$, \textit{i.e.} a tropical curve, and $\omega\in\Lambda^2M$ be a $2$-form. We consider the following variant of the $\omega$-problem.

\begin{prob}
How many rational tropical curves $h:\Gamma\rightarrow N_\RR$ of degree $\Delta$ have fixed family of moments $\mu_\omega$ and satisfy $h(e_0)\in\Xi_{e_0}$ ?
\end{prob}

We count the solutions to the problem with multiplicity $|\omega(n_{e_0},\delta)|B_\Gamma^{K_\omega}$, where $\delta$ is the slope of $\Xi_{e_0}$ at the intersection point with $\Gamma$:
$$\mathcal{B}_\Delta^\omega(\mu_\omega,\Xi_{e_0})=\sum_{\substack{\mathrm{\Gamma}=\mu_\omega \\ h(e_0)\in\Xi_{e_0} }} |\omega(n_{e_0},\delta)|B_\Gamma^{K_\omega}.$$
We have the following invariance result.

\begin{theo}
For a generic choice of $(\mu_\omega,\Xi_{e_0})$, the count does not depend on the choice of $\mu_\omega$ nor the choice of $\Xi_{e_0}$. It only depends on $\omega$ and the degree of $\Xi_{e_0}$. Moreover, its value is
$$\mathcal{B}_\Delta^\omega(\Xi_{e_0})=(\omega\cdot\Xi_{e_0})\mathcal{B}_\Delta^{\omega,e_0},$$
where $\omega\cdot\Xi_{e_0}$ is the intersection index between $\Xi_{e_0}$ and an affine hyperplane of slope $H_{e_0}$.
\end{theo}

\begin{proof}
The proof is immediate because due to Theorem \ref{theorem invariance moment problem}, for each intersection point between $\mathcal{H}_{e_0}$ and $\Xi_{e_0}$, the refined count is $|\omega(n_{e_0},\delta)|\mathcal{B}_\Delta^{\omega,e_0}$, the result follows by definition of the intersection index.
\end{proof}

\subsection{Constraints on marked points}

The other main direction to generalize the refined invariance theorems by staying in the realm of cycles constraints would be to allow the constraints to live in the main strata $N_\RR$. This would mean to allow unbounded ends with slope $0$ and impose cycle constraints on them.

The first problem that occurs is that the $\omega$-problem, which is at the core of the proof of invariance, does not allow unbounded ends to have slope $0$. If we are to go by the same proof as for Theorem \ref{theorem invariance general}, the first step is then to find an suitable enumerative problem that does so. For instance, the following variant of the $\omega$-problem. The idea is to add marked points that do not add any constraints for the curves.

Let $\varphi_1,\dots,\varphi_p\in M$ be $p$ non-zero linear forms, and let $\mathcal{H}_i\subset N_\RR$ be affine hyperplanes with respective slope $\ker\varphi_i$. Let as usual be $\omega$ a $2$-form, $\mu\in\mathfrak{M}$ a family of moments and $D\subset N_\RR/\gen{e_0}$, all chosen generically. We consider the following enumerative problem $\mathcal{P}\left(\omega,e_0,(\varphi_i)\right)$. We consider curves of degree $\Delta\sqcup\{0^p\}$, which means that curves have now $p$ new unbounded ends of slope $0$, which are called \textit{marked points} and are denoted by $i\in[\![1;p]\!]$.

\begin{prob}
\label{moment problem enhanced} $\mathcal{P}\left(\omega,e_0,(\varphi_i)\right)$
How many parametrized rational tropical curves $h:\Gamma\rightarrow N_\RR$ of degree $\Delta\sqcup\{0^p\}$ satisfy $\mathrm{mom}(\Gamma)=\mu$, $h(e_0)\in D$, and $h(i)\in\mathcal{H}_i$ ?
\end{prob}

\begin{rem}
\label{remark solutions moment problem enhanced}
The new additional conditions are kind of void for the following reason. Forgetting the marked points, each solution of $\mathcal{P}\left(\omega,e_0,(\varphi_i)\right)$ gives a solution of the usual $\omega$-problem $\mathcal{P}(\omega,e_0)$. Conversely, every solution to $\mathcal{P}\left(\omega,e_0,(\varphi_i)\right)$ is obtained from a solution to the standard $\omega$-problem $\mathcal{P}(\omega,e_0)$ by choosing intersection points with the hyperplanes $\mathcal{H}_i$. Such points always exist. Therefore, the solutions are morally the same solutions as for the $\omega$-problem. The new marked points on the curve just have to be chosen at the intersection points between the solutions to the $\omega$-problem and $\mathcal{H}_i$.
\end{rem}

Standard intersection theory ensures that the count of solutions with complex multiplicity yields an invariant. Using the algorithm to compute the complex multiplicity of a solution $h:\Gamma\rightarrow N_\RR$, one sees that the vertex corresponding to the marked point $i$ just contributes a factor $|\varphi_i(\delta_i^\Gamma)|$, where $\delta_i^\Gamma$ is the slope of $\Gamma$ at the marked point. Hence, the complex multiplicity tales the following form:
$$m_\Gamma^{\CC,\omega,e_0,(\varphi_i)}=\prod_i|\varphi_i(\delta_i^\Gamma)|\times m_\Gamma^{\CC,\omega,e_0}.$$
Using Remark \ref{remark solutions moment problem enhanced}, one obtains that by adding all the multiplicities over the curves solutions to problem \ref{moment problem enhanced},
$$\sum_{\Gamma}m_\Gamma^{\CC,\omega,e_0,(\varphi_i)}=\left(\prod_i\varphi_i\cdot\Delta\right) N_\Delta^{\CC,\omega,e_0},$$
where $\varphi_i\cdot\Delta=\sum_{n\in\Delta:\varphi_i(n)>0}\varphi_i(n)$ denotes the intersection index between a curve of degree $\Delta$ and a hyperplane of slope $\ker\varphi_i$. The combinatorial types where $\varphi_i(\delta_i^C)$ have thus zero complex multiplicity. They were previously part of the fourth type. They do not occur if $(\varphi_i)$ are chosen generically regarding $\Delta$.

We now get to the refined count of this enhanced $\omega$-problem.

\begin{prop}
The count of solutions to problem $\mathcal{P}\left(\omega,e_0,(\varphi_i)\right)$ with refined multiplicity $\prod_i|\varphi_i(\delta_i^\Gamma)| B_\Gamma^{K_\omega}$ yields an invariant which is equal to $\left(\prod_i\varphi_i\cdot\Delta\right)\mathcal{B}_\Delta^{\omega,e_0}$.
\end{prop}

\begin{proof}
The proof uses Remark \ref{remark solutions moment problem enhanced} which describes solutions to problem \ref{moment problem enhanced} in terms of solutions to the usual moment problem. The direct computations yields the result.
\end{proof}

We can then proceed, applying the same trick as in the proof of Theorem \ref{theorem invariance general}, to get a new invariance result. Let $L_e\subset N/\gen{n_e}$ for $e\in\Gamma_\infty^1$ and $L_i\subset N_i$ for $i\in[\![1;p]\!]$ be full sublattices such that
$$\sum_{z\in\Gamma_\infty^1}\mathrm{cork}L_e + \sum_{i=1}^p \mathrm{cork}L_i=|\Delta|+r+p-3 .$$
Let $\mathcal{L}_e$, $\mathcal{L}_i$ be affine subspaces of respective slopes $L_e,L_i$. We consider the following enumerative problem.

\begin{prob}
\label{problem genral interior constraints}
How many rational curves of degree $\Delta\sqcup\{0^p\}$ satisfy $h(e)\in\mathcal{L}_e$ and $h(i)\in\mathcal{L}_i$ ?
\end{prob}

\begin{rem}
The infinite constraints associated to the $\mathcal{L}_e$ can be forgotten by taking $L_e=N/\gen{n_e}$.
\end{rem}

Standard intersection theory yields that the use of the complex multiplicity $m_C^{\CC,(L_e),(L_i)}$ leads to an invariant. Meanwhile, the use of refined multiplicity gives an invariant as follows. Let $\omega,\varphi_i$ be such that
$$m_C^{\CC,(L_e),(L_i)}=0\Rightarrow m_C^{\CC,\omega,e_0,(\varphi_i)}=0.$$

\begin{theo}
There exists some signs $\varepsilon_\Gamma$ such that the refined count of solutions to problem \ref{problem genral interior constraints} using multiplicity $\varepsilon_C\prod_i|\varphi_i(\delta_i^C)| B_C^{K_\omega}$ leads to an invariant.
\end{theo}

\begin{proof}
The proof is identical to the proof of Theorem \ref{theorem invariance general}.
\end{proof}

\begin{rem}
One gets many new invariants depending on the choice of the linear forms $\varphi_i$. Clearly, the invariants are $1$-homogeneous in the linear forms $\varphi_i$. Yet, it would be interesting to study the various obtained invariants.
\end{rem}

As before, there is a subtle condition for the existence of refined invariants, which is the existence of $\omega,(\varphi_i)$. To see that this generalization is still of interest, notice that the class of enumerative problems for which constraints on the unbounded ends and marked points are of codimension at least $1$ provide a family of examples for which a generic choice of $\omega$ and $\varphi_i$ satisfy the assumption, leading to refined invariants in $\ZZ[\Lambda^2 N]$.

As in the case of tropical cycle constraints, the appearance of factors $\prod_i|\varphi_i(\delta_i^C)|$ makes the new multiplicity lose its relationship to the real curves tropicalizing to the tropical curve.

\section{Examples of refined invariants}

We now give a few examples, limited for the two following computational technicalities: first, we do not have yet a way of solving efficiently either of the enumerative problems $\mathcal{P}(\omega,e_0)$ and $\mathcal{P}(L_e)$, secondly, in the case of $\mathcal{P}(L_e)$, the finding of a $2$-form $\omega$ might be difficult.

\begin{expl}
First, let assume that $\omega=\varphi\wedge\psi$ is a generic $2$-form of rank $1$, where $\varphi,\psi\in M$. Then, the $2$-form has a $(r-2)$-dimensional kernel $\ker\omega=\ker\varphi\cap\ker\psi$ which is contained in every $\ker\iota_{n_e}\omega$. Therefore, the translate of a curve $h:\Gamma\rightarrow N_\RR$ by a vector in $\ker\omega$ has the same moments. Hence, solving the $\omega$-problem in $N_\RR$ can be done by solving the $\omega$-problem in $N_\RR/\ker\omega$ and then lifting the solutions to $N_\RR$: a solution $h:\Gamma\rightarrow N_\RR$ to $\mathcal{P}(\omega,e_0)$ gives a solution to $\mathcal{P}(\overline{\omega})$ to the $\omega$-problem in $N_\RR/\ker\omega$ by composing with the quotient map $N_\RR\rightarrow N_\RR/\ker\omega$, which is endowed with the quotient form $\overline{\omega}$. Conversely, any solution in the quotient can uniquely be lifted to a solution in $N_\RR$ such that $h(e_0)$ belongs to the line $D$ chosen for the enumerative problem. Therefore, we reduce to the planar case and the computation can be done using the recursive formula from \cite{blomme2019caporaso}. In particular, as the solution are determined uniquely up to translation by the solution in the quotient, for such $2$-forms, the value of $\mathcal{B}_\Delta^{\omega,e_0}$ does not depend on the choice of $e_0$.
\end{expl}

In all the following examples, we now consider tropical lines in $\RR^4$, \textit{i.e.} curves of degree $\Delta=\{e_1,e_2,e_3,e_4,e_5\}$, where $(e_1,e_2,e_3,e_4)$ is the canonical basis of $\ZZ^4$ and $e_5=-\sum_1^4 e_i$. Up to a change of basis of the lattice $N=\ZZ^4$ and maybe a change of lattice, this includes all cases of degree $\Delta$ of cardinal $5$ which span a $4$-dimensional space. The moduli space $\mathcal{M}(\Delta,\RR^4)$ is of dimension $6$, in bijection with $\mathcal{M}_{0,5}\times\RR^4$. The moduli space $\mathcal{M}_{0,5}$ is the cone over the Petersen graph. It has $15$ cones of dimension $2$ corresponding to the $15$ top-dimensional combinatorial types, and $10$ rays, which are the walls of our enumerative problem. This small number of possibilities allows one to solve the enumerative problems by brute force using a computer, providing us with the following examples, which we use to illustrate the various results of the paper more than to provide useful values. We identify $\ZZ[\Lambda^2\ZZ^4]$ with $\ZZ[q_{ij}^{\pm 1}]_{1\leqslant i<j\leqslant 4}$ through its canonical basis: $q_{ij}=q^{e_i\wedge e_j}$.

\begin{expl}
Let us start with a simple example. We consider the $\omega$-problem with the generic $2$-form whose matrix in the canonical basis is given by
$$[\omega_1]=\begin{pmatrix}
0 & -68 & -53 &  86 \\
68 & 0 & 46 & -43 \\
53 & -46 & 0 & 30 \\
-86 & 43 & -30 & 0 \\
\end{pmatrix}.$$
The genericity means that the $2$-form does not cancel any $\pi_V$ for $V\in\mathcal{V}_\Delta$. The $\omega$-problem $\mathcal{P}(\omega_1,e_1)$ with chosen unbounded end directed by $e_1$ gives
\begin{align*}
\mathcal{B}_\Delta^{\omega_1,e_1} = &
-q_{12}q_{13}q_{14}q_{23}q_{24}q_{34} + \frac{1}{q_{12}q_{13}q_{14}q_{23}q_{24}q_{34}} && + 2\frac{q_{12}q_{13}q_{14}q_{23}q_{24}}{q_{34}}- 2\frac{q_{34}}{q_{12}q_{13}q_{14}q_{23}q_{24}} \\
& - \frac{q_{12}q_{13}q_{14}q_{23}}{q_{24}q_{34}} + \frac{q_{24}q_{34}}{q_{12}q_{13}q_{14}q_{23}}
&& + \frac{q_{12}q_{13}q_{14}q_{34}}{q_{23}q_{24}} - \frac{q_{23}q_{24}}{q_{12}q_{13}q_{14}q_{34}} \\
& - \frac{q_{12}q_{13}}{q_{14}q_{23}q_{24}q_{34}} + \frac{q_{14}q_{23}q_{24}q_{34}}{q_{12}q_{13}}
&& - \frac{q_{12}q_{14}q_{24}q_{34}}{q_{13}q_{23}} + \frac{q_{13}q_{23}}{q_{12}q_{14}q_{24}q_{34}} \\
& + \frac{q_{12}q_{34}}{q_{13}q_{14}q_{23}q_{24}} - \frac{q_{13}q_{14}q_{23}q_{24}}{q_{12}q_{34}}. & & \\
\end{align*}
\end{expl}

\begin{expl}
Similarly, we now consider the $2$-form whose matrix is
$$[\omega_2]=\begin{pmatrix}
0 & 94 & 23 &  21 \\
-94 & 0 & 86 & 11 \\
23 & -86 & 0 & -27 \\
-21 & -11 & 27 & 0 \\
\end{pmatrix}.$$
For different choices of the unbounded end with the plane constraint, one has
\begin{align*}
\mathcal{B}_\Delta^{\omega_2,e_1}=\mathcal{B}_\Delta^{\omega_2,e_3}= &
\frac{q_{12}q_{13}q_{14}q_{23}q_{24}}{q_{34}}- \frac{q_{34}}{q_{12}q_{13}q_{14}q_{23}q_{24}} &&
- \frac{q_{12}q_{13}q_{23}}{q_{14}q_{24}q_{34}} + \frac{q_{14}q_{24}q_{34}}{q_{12}q_{13}q_{23}} \\
& - \frac{q_{12}q_{13}q_{14}}{q_{23}q_{24}q_{34}} + \frac{q_{23}q_{24}q_{34}}{q_{12}q_{13}q_{14}}
&&
- \frac{q_{12}q_{14}q_{24}q_{34}}{q_{13}q_{23}} + \frac{q_{13}q_{23}}{q_{12}q_{14}q_{24}q_{34}} \\
& + \frac{q_{12}q_{13}q_{14}q_{34}}{q_{23}q_{24}}- \frac{q_{23}q_{24}}{q_{12}q_{13}q_{14}q_{34}}
&&
+ \frac{q_{12}}{q_{13}q_{14}q_{23}q_{24}q_{34}} - \frac{q_{13}q_{14}q_{23}q_{24}q_{34}}{q_{12}} .\\
\end{align*}

\begin{align*}
\mathcal{B}_\Delta^{\omega_2,e_2}= &
\frac{q_{12}q_{13}q_{14}q_{23}q_{24}}{q_{34}}- \frac{q_{34}}{q_{12}q_{13}q_{14}q_{23}q_{24}} &&
- \frac{q_{12}q_{13}q_{23}}{q_{14}q_{24}q_{34}} + \frac{q_{14}q_{24}q_{34}}{q_{12}q_{13}q_{23}} \\
& - \frac{q_{12}q_{13}q_{14}}{q_{23}q_{24}q_{34}} + \frac{q_{23}q_{24}q_{34}}{q_{12}q_{13}q_{14}}
&&
- \frac{q_{12}q_{14}q_{24}q_{34}}{q_{13}q_{23}} + \frac{q_{13}q_{23}}{q_{12}q_{14}q_{24}q_{34}}  \\
& + \frac{q_{12}q_{13}}{q_{14}q_{23}q_{24}q_{34}}- \frac{q_{14}q_{23}q_{24}q_{34}}{q_{12}q_{13}}
&&
+ \frac{q_{12}q_{14}q_{34}}{q_{13}q_{23}q_{24}} - \frac{q_{13}q_{23}q_{24}}{q_{12}q_{14}q_{34}} .\\
\end{align*}

In particular, we see that in general the value of the invariant $\mathcal{B}_\Delta^{\omega,e_0}$ depends on the choice of $e_0$ since in that case the invariants $\mathcal{B}_\Delta^{\omega_2,e_1}$ and $\mathcal{B}_\Delta^{\omega_2,e_2}$ are different: the last row of their expression are differents.
\end{expl}

\begin{expl}
We now look at the $2$-form $\omega_0$ given by
$$[\omega_0]=\begin{pmatrix}
0 & 0 & 86 &  -20 \\
0 & 0 & -4 & -22 \\
-86 & 4 & 0 & -56 \\
20 & 22 & 56 & 0 \\
\end{pmatrix}.$$
This $2$-form is non-generic since it has value zero at the vertex formed by the meeting of the unbounded ends directed by $e_1$ and $e_2$. The invariant lives in $\ZZ[\Lambda^2\ZZ^4/\gen{e_1\wedge e_2}]$, which is identified with $\ZZ[q_{ij}]_{(i,j)\neq(1,2)}$. One gets
\begin{align*}
\mathcal{B}_\Delta^{\omega_0,e_1}= &
q_{13}q_{14}q_{23}q_{24}q_{34} - \frac{1}{q_{13}q_{14}q_{23}q_{24}q_{34}}
&&
-\frac{q_{13}q_{14}q_{34}}{q_{23}q_{24}} + \frac{q_{23}q_{24}}{q_{13}q_{14}q_{34}} \\ 
& + \frac{q_{14}q_{34}}{q_{13}q_{23}q_{24}} - \frac{q_{13}q_{23}q_{24}}{q_{14}q_{34}}
&&
+ \frac{q_{13}}{q_{14}q_{23}q_{24}q_{34}} - \frac{q_{14}q_{23}q_{24}q_{34}}{q_{13}}  ,\\
\end{align*}
There are two ways to deform slightly the $2$-form $\omega_0$, getting the following $2$-forms:
$$[\omega_+]=\begin{pmatrix}
0 & 1 & 86 &  -20 \\
-1 & 0 & -4 & -22 \\
-86 & 4 & 0 & -56 \\
20 & 22 & 56 & 0 \\
\end{pmatrix} \text{ and }[\omega_-]=\begin{pmatrix}
0 & -1 & 86 &  -20 \\
1 & 0 & -4 & -22 \\
-86 & 4 & 0 & -56 \\
20 & 22 & 56 & 0 \\
\end{pmatrix}.$$
In fact, $\omega_0$ belongs to a facet of the fan $\Omega_\Delta$, and $\omega_\pm$ belongs to each of the adjacent maximal cones. They have the following invariants:
\begin{align*}
\mathcal{B}_\Delta^{\omega_+,e_1}= &
\frac{q_{13}q_{14}q_{23}q_{24}q_{34}}{q_{12}} - \frac{q_{12}}{q_{13}q_{14}q_{23}q_{24}q_{34}}
&&
-\frac{q_{12}q_{13}q_{14}q_{34}}{q_{23}q_{24}} + \frac{q_{23}q_{24}}{q_{12}q_{13}q_{14}q_{34}} \\ 
& + \frac{q_{12}q_{14}q_{34}}{q_{13}q_{23}q_{24}} - \frac{q_{13}q_{23}q_{24}}{q_{12}q_{14}q_{34}}
&&
+ \frac{q_{12}q_{13}}{q_{14}q_{23}q_{24}q_{34}} - \frac{q_{14}q_{23}q_{24}q_{34}}{q_{12}q_{13}}  ,\\
\end{align*}

\begin{align*}
\mathcal{B}_\Delta^{\omega_-,e_1}=&
q_{12}q_{13}q_{14}q_{23}q_{24}q_{34} - \frac{1}{q_{12}q_{13}q_{14}q_{23}q_{24}q_{34}}
&&
- \frac{q_{12}q_{13}q_{14}q_{34}}{q_{23}q_{24}}+ \frac{q_{23}q_{24}}{q_{12}q_{13}q_{14}q_{34}} \\
&
+ \frac{q_{12}q_{14}q_{34}}{q_{13}q_{23}q_{24}} - \frac{q_{13}q_{23}q_{24}}{q_{12}q_{14}q_{34}}
&&
+ \frac{q_{12}q_{13}}{q_{14}q_{23}q_{24}q_{34}}- \frac{q_{14}q_{23}q_{24}q_{34}}{q_{12}q_{13}} \\
&
- \frac{q_{12}q_{14}q_{24}q_{34}}{q_{13}q_{23}} + \frac{q_{13}q_{23}}{q_{12}q_{14}q_{24}q_{34}}
&&
- \frac{q_{12}q_{13}q_{23}}{q_{14}q_{24}q_{34}} + \frac{q_{14}q_{24}q_{34}}{q_{12}q_{13}q_{23}}.\\
\end{align*}

In particular, one can see that although $\mathcal{B}_\Delta^{\omega_+,e_1}\neq\mathcal{B}_\Delta^{\omega_-,e_1}$, they both project to $\mathcal{B}_\Delta^{\omega_0,e_1}$ by the quotient map $\ZZ[\Lambda^2\ZZ^4]\rightarrow\ZZ[\Lambda^2\ZZ^4/\gen{e_1\wedge e_2}]$, which in our notations identifies with the map $\ZZ[q_{ij}]\xrightarrow{q_{12}=1}\ZZ[q_{ij}]_{(i,j)\neq(1,2)}$: for $\mathcal{B}_\Delta^{\omega_+,e_1}$ it is immediate, and for $\mathcal{B}_\Delta^{\omega_-,e_1}$, the first two rows give the value of $\mathcal{B}_\Delta^{\omega_0,e_1}$ while the last rows cancels when $q_{12}$ goes to $1$.
\end{expl}

\begin{expl}
We now give an example of refined invariant in the general case. We consider the enumerative problem with the following constraints:
\begin{itemize}
\item The unbounded end directed by $e_1$ lies in affine subspace of slope
$$L_{e_1}=\ker\begin{pmatrix}
0 & 3780 & -315 & -2543 \\
\end{pmatrix}\cap\ker\begin{pmatrix}
0 & -6958 & 7243 & 3904 \\
\end{pmatrix}.$$
\item The unbounded end directed by $e_2$ lies in an affine hyperplane of slope
$$L_{e_2}=\ker\begin{pmatrix}
-25 & 0 & -16 & -72 \\
\end{pmatrix}.$$
\item The unbounded end directed by $e_3$ lies in an affine hyperplane of slope
$$L_{e_3}=\ker\begin{pmatrix}
-4387 & 564 & 0 & 2857 \\
\end{pmatrix}.$$
\item $e_4$ lies in an affine hyperplane of slope directed by $L_{e_4}=\ker\begin{pmatrix}
-720 & -843 & -718 & 0 \\
\end{pmatrix}$.
\item $e_5$ lies in an affine hyperplane of slope directed by $L_{e_5}=\ker\begin{pmatrix}
-1091 & -562 & 653 & 1000 \\
\end{pmatrix}$.
\end{itemize}
These particular constraints are in the family of problems described in example \ref{example wide class}. Thus, a generic choice of $\omega$ satisfies the condition of Theorem \ref{theorem invariance general}. We choose
$$[\omega_3]=\begin{pmatrix}
0 & 20 & -51 & 38 \\
-20 & 0 & 89 & 13 \\
51 & 4 & 0 & -24 \\
-38 & -13 & 24 & 0 \\
\end{pmatrix}.$$
The rule of sign and computation of the invariants then gives
\begin{align*}
\mathcal{B}_\Delta^{(L_e)}= & 
\frac{q_{12}q_{13}q_{14}q_{23}q_{24}}{q_{34}}- \frac{q_{34}}{q_{12}q_{13}q_{14}q_{23}q_{24}} 
 && - \frac{q_{12}q_{13}q_{14}q_{24}q_{34}}{q_{23}} + \frac{q_{23}}{q_{12}q_{13}q_{14}q_{24}q_{34}}  \\
& + \frac{q_{12}q_{13}q_{14}q_{34}}{q_{23}q_{24}}- \frac{q_{23}q_{24}}{q_{12}q_{13}q_{14}q_{34}}
&& - \frac{q_{12}q_{13}q_{23}}{q_{14}q_{24}q_{34}} + \frac{q_{14}q_{24}q_{34}}{q_{12}q_{13}q_{23}}  \\
& - \frac{q_{12}q_{14}q_{34}}{q_{13}q_{23}q_{24}}+ \frac{q_{13}q_{23}q_{24}}{q_{12}q_{14}q_{34}}
&& + \frac{q_{12}q_{34}}{q_{13}q_{14}q_{23}q_{24}} - \frac{q_{13}q_{14}q_{23}q_{24}}{q_{12}q_{34}}  . \\
\end{align*}
\end{expl}

\begin{expl}
Last, we give an example where a generic choice of $\omega$ does not satisfy the assumption of Theorem \ref{theorem invariance general}. We consider the enumerative problem of lines passing through two points which are located on the unbounded ends directed by $e_1$ and $e_5$. Notice that by standard geometry, there is always exactly one line passing through two points. Hence, the complex multiplicity is either $0$ or $1$, and combinatorial type can provide a solution only one at a time.

The combinatorial type where some of the unbounded ends directed by $e_2$, $e_3$ and $e_4$ are adjacent to a common vertex have complex multiplicity $0$. One can check that the other combinatorial types, of the form $1i//j//k5$ have complex multiplicity $1$. This notation denotes the combinatorial type where unbounded ends directed by $e_1$ and $e_i$ are adjacent to the same vertex, as well as $e_k$ and $e_5$. In order to apply Theorem \ref{theorem invariance general}, we can choose a $2$-form $\omega$ that cancels on $e_2\wedge e_3$, $e_2\wedge e_4$ and $e_3\wedge e_4$. For instance, take
$$[\omega_4]=\begin{pmatrix}
0 & 1 & 1 & 1 \\
-1 & 0 & 0 & 0 \\
-1 & 0 & 0 & 0 \\
-1 & 0 & 0 & 0 \\
\end{pmatrix}.$$
One has $K_{\omega_4}=\gen{e_2\wedge e_3,e_2\wedge e_4,e_3\wedge e_4}$. Thus, the invariant lives in $\ZZ[\Lambda^2\ZZ^4/K_{\omega_4}]\simeq\ZZ[q_{12},q_{13},q_{14}]$. The invariance stated in Theorem \ref{theorem invariance general} means that every tropical line with non-zero complex multiplicity has the same refined multiplicity. For instance,
\begin{align*}
\mathcal{B}_\Delta^{K_{\omega_4}}=B_{12//3//45}^{K_{\omega_4}} = & (q_{12}-q_{12}^{-1})(q_{13}q_{23}-q_{13}^{-1}q_{23}^{-1})(q_{14}q_{24}q_{34}-q_{14}^{-1}q_{24}^{-1}q_{34}^{-1}) \\
& =(q_{12}-q_{12}^{-1})(q_{13}-q_{13}^{-1})(q_{14}-q_{14}^{-1})\in\ZZ[q_{12},q_{13},q_{14}].\\
\end{align*}
\begin{align*}
\mathcal{B}_\Delta^{K_{\omega_4}}=B_{13//2//45}^{K_{\omega_4}} = & (q_{13}-q_{13}^{-1})\left(\frac{q_{12}}{q_{23}}-\frac{q_{23}}{q_{12}}\right)(q_{14}q_{24}q_{34}-q_{14}^{-1}q_{24}^{-1}q_{34}^{-1}) \\
& =(q_{12}-q_{12}^{-1})(q_{13}-q_{13}^{-1})(q_{14}-q_{14}^{-1})\in\ZZ[q_{12},q_{13},q_{14}].\\
\end{align*}
Recall that the term of each product are obtained as $q^{\pi_V}-q^{-\pi_V}$ where $\omega_4(\pi_V)>0$. For instance, for the midle vertex of the combinatorial type $13//2//45$, one has $\omega_4((e_1+e_3)\wedge e_2)=1>0$, thus,
$$q^{(e_1+e_3)\wedge e_2}=q^{e_1\wedge e_2}q^{-e_2\wedge e_3}=\frac{q_{12}}{q_{23}}.$$
\end{expl}

\section*{Appendix: computation of the complex multiplicity}

We here prove that the algorithm described in section \ref{section complex multiplicity} computes the complex multiplicity. This result also appears in a paper of T.~Mandel and H.~Ruddat \cite{mandel2019tropical} in a slightly different context. We include a proof in our setting for sake of completeness.

\begin{lem}
The value of the multiple does not depend on the chosen sink.
\end{lem}

\begin{proof}
We prove that the obtained value does not change if we replace the sink $V$ by one of its neighbors. Let $V$ be a vertex, $W$ one of its neighbors, and $E$ the edge between them, directed by $n$. Let $\rho_1,\dots,\rho_s$ be the polyvectors of the edges adjacent to $V$ different from $E$, and $\rho'_1,\dots,\rho'_{s'}$ the polyvectors associated to the edges adjacent to $W$ different from $E$. The computation leads to the two following results:
\begin{itemize}
\item If $V$ is the sink, we get
$$\rho_1\wedge\cdots\wedge\rho_s\wedge\iota_n(\rho'_1\wedge\cdots\wedge\rho'_{s'}).$$
\item If $W$ is the sink, we get
$$\iota_n(\rho_1\wedge\cdots\wedge\rho_s)\wedge\rho'_1\wedge\cdots\wedge\rho'_{s'}.$$
\end{itemize}
Therefore, the equality (up to sign) comes from the fact that $\rho_1\wedge\cdots\wedge\rho_s\wedge\rho'_1\wedge\cdots\wedge\rho'_{s'}$ is $0$ since it is in $\Lambda^{r+1}M=\{0\}$, and that $\iota_n$ is a derivation on $\Lambda^\bullet M$.
\end{proof}

\begin{theo}
\label{Theorem calcul complex multiplicity tropical curve}
The value obtained by the preceding algorithm is equal to the complex multiplicity $m_\Gamma^\CC$.
\end{theo}

\begin{proof}
We make an induction on the number of vertices of the curve $\Gamma$, and cut the branches one by one.
\begin{itemize}[label=-]
\item If the curve has just one vertex, let $\rho_i=m_{i1}\wedge\cdots\wedge m_{ir_i}$ be the polyvectors associated to the lattices $L_i$ for the $s$ unbounded edges of the curve. The evaluation matrix of $\mathrm{ev}^{(L_e)}$, denoted by $[f]$, has the following form:
$$[\mathrm{ev}^{(L_e)}]=\begin{pmatrix}
m_{11}\\
\vdots\\
m_{1r_1}
\vdots\\
m_{sr_s}\\
\end{pmatrix}.$$
Therefore, we have 
$$\mathrm{det}\ \mathrm{ev}^{(L_e)}=m_{11}\wedge\cdots\wedge m_{1r_1}\wedge\cdots\wedge m_{sr_s}=\rho_1\wedge\cdots\wedge \rho_s.$$
\item If $\Gamma$ has more than one vertex, let $V$ be a vertex adjacent to only unbounded ends associated to polyvectors $\rho_1,\dots,\rho_s$, and  one unique neighbor vertex $W$. We keep the same notations $\rho_i=m_{i1}\wedge\cdots\wedge m_{ir_i}$. We choose a basis of the cone associated to the combinatorial type of $\Gamma$ consisting of the canonical basis of $\RR_{\geqslant 0}^{m-3-\mathrm{ov}(\Gamma)}$, and the $N_\RR$ factor corresponding to the position of $W$. Then, the matrix of $\mathrm{ev}^{(L_e)}$ has the following form:
$$[\mathrm{ev}^{(L_e)}]=\begin{pmatrix}
* & \cdots & *    &   * & \cdots & * & 0 \\
\vdots & \ddots & \vdots    &   \vdots & \ddots & \vdots & \vdots \\
* & \cdots & *    &   * & \cdots & * & 0 \\

 & m_{11} &     &   0 & \cdots & 0 & \langle m_{11},n\rangle \\
  & \vdots &     &   \vdots & \ddots & \vdots & \vdots \\
  & m_{sr_s} &     &   0 & \cdots & 0 & \langle m_{sr_s},n\rangle \\
\end{pmatrix},$$
where the first columns correspond to the evaluation of the vertex $W$, the last column to the evaluation corresponding to the edge between $V$ and $W$. We make a development with respect to the last column. We get the following result:
$$\mathrm{det}\ \mathrm{ev}^{(L_e)}=\sum_{i=1}^s\sum_{j=1}^{r_i} (-1)^\bullet \langle m_{ij},n\rangle
\left|
\begin{matrix}
* & \cdots & *    &   * & \cdots & *  \\
\vdots & \ddots & \vdots    &   \vdots & \ddots & \vdots  \\
* & \cdots & *    &   * & \cdots & *  \\

 & m_{11} &     &   0 & \cdots & 0  \\
  & \vdots &     &   \vdots & \ddots & \vdots  \\
  & \hat{m}_{ij} &     &   0 & \cdots & 0  \\
  & \vdots &     &   \vdots & \ddots & \vdots \\
  & m_{sr_s} &     &   0 & \cdots & 0 \\
\end{matrix}
\right| .$$
Each determinant in the sum is the determinant of the evaluation matrix $\mathrm{ev}^{(L_e)}$ for a tropical curve where the vertex $V$ is deleted and the edge between $V$ and $W$ replaced by an unbounded end, associated with a constraint having polyvector $(-1)^\bullet \langle m_{ij}, n\rangle m_{11}\wedge\cdots\wedge\hat{m}_{ij}\wedge\cdots\wedge m_{sr_s}$. The sum of these vectors is precisely $\iota_n(\rho_1\wedge\cdots\wedge\rho_s)$. Hence, the result follows by induction.
\end{itemize}
\end{proof}

\section*{Appendix: invariance for the complex multiplicity}

We consider parametrized rational tropical curves of degree $\Delta$ in $N_\RR$. We allow vectors of $\Delta$ to be $0$, so that the unbounded ends associated to these vectors are marked points on the curves. For each unbounded end $e$, let $L_e$ be a primitive sublattice of $N/\langle n_e\rangle$, and let $l_e$ be its corank. We denote by $L_e^\RR=L_e\otimes\RR$. Notice that if $e$ corresponds to a marked point, $L_e$ is a sublattice of $N$. Let $\mathcal{L}_e$ be a generic affine space in $N_\RR/\langle n_i\rangle$ with slope $L_e^\RR$. This amounts to the choice of a point $\lambda_e$ in $N_\RR/\left(\langle n_e\rangle\oplus L_e^\RR\right)$. We prove the invariance of the solutions to problem $\mathcal{P}(L_e)$ using the complex multiplicity $m^{\CC,(L_e)}$.

\begin{lem}
For a generic choice of $(\lambda_e)$, there is a finite number of solutions to $\mathcal{P}(L_e)$. Moreover, these solutions are trivalent.
\end{lem}

\begin{proof}
There is a finite number of cones in the moduli space $\mathcal{M}_0(\Delta,N_\RR)$. Moreover, on each cone, \textit{i.e.} each combinatorial type, $\mathrm{ev}^{(L_e)}$ is linear. If the restriction of $\mathrm{ev}^{(L_e)}$ to this cone is injective, the combinatorial type contributes at most one solution.

The point $(\lambda_e)$ can always be chosen outside the image of the non top-dimensional cones, since these images are included in proper subspaces of $\prod_e Q_e^\RR$. For a combinatorial type of top-dimension, as by assumption its dimension is equal to the dimension of $\prod_e Q_e^\RR$, if $\mathrm{ev}^{(L_e)}$ is not injective, it is not surjective either, and its image is thus contained in a proper subspace. Finally, a generic choice of $(\lambda_e)$ is chosen outside the image of the non top-dimensional cones and the cones where $\mathrm{ev}^{(L_e)}$ is not injective.

For such a choice of $(\lambda_e)$, there is a finite number of solutions, and the corresponding curves are trivalent since they belong to top dimensional cones of $\mathcal{M}_0(\Delta,N_\RR)$.
\end{proof}

Now, let $(\lambda_e)$ be chosen generically. We set
$$N_\Delta(\mathcal{L}_e)=\sum_{\mathrm{ev}^{(L_e)}(\Gamma,h)=(\lambda_e)}m_\Gamma^\CC.$$

Notice that if the subspaces $\mathcal{L}_e$ are chosen generically, there is no curve with zero multiplicity that contributes to the sum, otherwise $\mathrm{ev}^{(L_e)}$ is not injective nor surjective on the cone corresponding to the combinatorial type of the curve, and thus $(\lambda_e)$ has been chosen out of its image.

\begin{theo}
\label{proposition invariance tropical count with complex multiplicities}
The value of $N_\Delta(\mathcal{L}_e)$ only depends on the slopes $L_e$ of the affine subspaces and not their specific choice as long as it is generic.
\end{theo}

This invariant is thus denoted by $N_\Delta(L_e)$.

\begin{proof}
As many proofs of tropical invariance, the proof goes by the study of local invariance at the walls of the tropical moduli space. The proof is similar to Proposition $4.4$ in \cite{gathmann2008kontsevich}. We proceed in two steps: first, we show that the sum of the determinants of the composed evaluation maps around a wall is zero, and then we show that the sign of these determinants characterizes the existence of solutions, thus proving the local invariance.

\begin{itemize}[label=$\bullet$]
\item We consider the wall associated to a quadrivalent vertex in the tropical curve. Let the adjacent edges be denoted by the indices $1,2,3,4$. The three adjacent combinatorial types are determined by the splitting of the quadrivalent vertex into two trivalent vertices. These possibilities are denoted by $12//34$, $13//24$ and $14//23$. The cone in the moduli space $\mathcal{M}_0(N_\RR,\Delta)$ corresponding to each combinatorial type is the quadrant of the vector space $N_\RR\times\RR^{m-4}\times\RR$, consisting of the points with positive coordinates on the $\RR$ entries, where the $N_\RR$ factor corresponds to the vertex $V$ adjacent to the edge $1$, the $\RR^{m-4}$ has canonical basis indexed by the edges of the curve with the quadrivalent vertex, which are common to all curves in the adjacent combinatorial types, and the $\RR$ factor corresponds to the length of the edge resulting from the splitting of the quadrivalent vertex.\\

Let $v_j$ be a directing vector of the edge $j$, oriented outward the quadrivalent vertex. For each marked point or unbounded end $i$, associated to a constraint $L_i$, let $m_{i1},\dots,m_{ir_i}$ be linear forms defining the  sublattice $L_i$: $L_i=\bigcap_{j=1}^{r_i}\ker m_{ij}$. Then, for the combinatorial type $12//34$, the matrix of the composed evaluation map $\mathrm{ev}^{(L_e)}$ takes the following form:

$$\begin{array}{c|c|c|c|c|c|c|c}
   12//34               & N_\RR & \RR^{m-8} & 1 & 2 & 3 & 4 & \RR \\ \hline
\text{behind }1 & m_{ij}  & 0 \text{ or } \langle m_{ij},v\rangle &  \langle m_{ij},v_1\rangle & 0 & 0 & 0 & 0     \\
\text{behind }2 & m_{ij}  & 0 \text{ or } \langle m_{ij},v\rangle &  0 & \langle m_{ij},v_2\rangle & 0 & 0 &  0    \\
\text{behind }3 & m_{ij}  & 0 \text{ or } \langle m_{ij},v\rangle &  0 & 0 & \langle m_{ij},v_3\rangle & 0 &  \langle m_{ij},v_1+v_2\rangle     \\
\text{behind }4 & m_{ij}  & 0 \text{ or } \langle m_{ij},v\rangle &  0 & 0 & 0 & \langle m_{ij},v_4\rangle &  \langle m_{ij},v_1+v_2\rangle \\  
\end{array}.$$

The columns are separated according to the decomposition of the moduli space as $N_\RR\times\RR^{m-4}\times\RR$. Moreover, we separate the coordinates corresponding to the lengths of the edges  $1,2,3,4$, assuming they are bounded edges. The rows are separated according to whether which of the four edges $1,2,3,4$ is on the shortest path between the vertex $V$ and the unbounded end or marked point $i$. For each unbounded end or marked point $i$, we evaluate the linear forms $(m_{ij})_j$. For each edge $e$ directed by $v$, the evaluation for the unbounded end or marked point $i$ is $\langle m_{ij},v\rangle$ if the edge $e$ is part of the shortest path between $V$ and unbounded end or marked point $i$, otherwise it is $0$. This fact allows us to complete the middle entries of the matrix. The same rule apply for the entries of the last columns. The matrices for the combinatorial types $13//24$ and $14//23$ are respectively

$$\begin{array}{c|c|c|c|c|c|c|c}
 13//24                 & N_\RR & \RR^{m-8} & 1 & 2 & 3 & 4 & \RR \\ \hline
\text{behind }1 & m_{ij}  & 0 \text{ or } \langle m_{ij},v\rangle & \langle m_{ij},v_1\rangle & 0 & 0 & 0 &  0     \\
\text{behind }2 & m_{ij}  & 0 \text{ or } \langle m_{ij},v\rangle &  0 & \langle m_{ij},v_2\rangle & 0 & 0 &  \langle m_{ij},v_1+v_3\rangle    \\
\text{behind }3 & m_{ij}  & 0 \text{ or } \langle m_{ij},v\rangle &  0 & 0 & \langle m_{ij},v_3\rangle & 0 &  0     \\
\text{behind }4 & m_{ij}  & 0 \text{ or } \langle m_{ij},v\rangle &  0 & 0 & 0 & \langle m_{ij},v_4\rangle &  \langle m_{ij},v_1+v_3\rangle \\  
\end{array},$$

$$\begin{array}{c|c|c|c|c|c|c|c}
  14//23                & N_\RR & \RR^{m-8} & 1 & 2 & 3 & 4 & \RR \\ \hline
\text{behind }1 & m_{ij}  & 0 \text{ or } \langle m_{ij},v\rangle & \langle m_{ij},v_1\rangle & 0 & 0 & 0 &  0     \\
\text{behind }2 & m_{ij}  & 0 \text{ or } \langle m_{ij},v\rangle &  0 & \langle m_{ij},v_2\rangle & 0 & 0 &  \langle m_{ij},v_1+v_4\rangle    \\
\text{behind }3 & m_{ij}  & 0 \text{ or } \langle m_{ij},v\rangle &  0 & 0 & \langle m_{ij},v_3\rangle & 0 &  \langle m_{ij},v_1+v_4\rangle     \\
\text{behind }4 & m_{ij}  & 0 \text{ or } \langle m_{ij},v\rangle &  0 & 0 & 0 & \langle m_{ij},v_4\rangle &  0 \\  
\end{array}.$$

We make the sum of the three determinants for the  three adjacent combinatorial types, and use the linearity with respect to the last column, since all the other columns are equal. We get

$$\begin{array}{c|c|c|c|c|c|c|c}
                 & N_\RR & \RR^{m-8} & 1 & 2 & 3 & 4 & \RR \\ \hline
\text{behind }1 & m_{ij}  & 0 \text{ or } \langle m_{ij},v\rangle & \langle m_{ij},v_1\rangle & 0 & 0 & 0 &  0     \\
\text{behind }2 & m_{ij}  & 0 \text{ or } \langle m_{ij},v\rangle &  0 & \langle m_{ij},v_2\rangle & 0 & 0 &    \langle m_{ij},2v_1+v_3+v_4\rangle    \\
\text{behind }3 & m_{ij}  & 0 \text{ or } \langle m_{ij},v\rangle &  0 & 0 & \langle m_{ij},v_3\rangle & 0 &    \langle m_{ij},2v_1+v_2+v_4\rangle     \\
\text{behind }4 & m_{ij}  & 0 \text{ or } \langle m_{ij},v\rangle &  0 & 0 & 0 & \langle m_{ij},v_4\rangle &    \langle m_{ij},2v_1+v_2+v_3\rangle \\  
\end{array}.$$

Using a combination of the columns corresponding to $N_\RR$ applied to $v_1$, and the balancing condition $v_1+v_2+v_3+v_4=0$, we get

$$\begin{array}{c|c|c|c|c|c|c|c}
                 & N_\RR & \RR^{m-8} & 1 & 2 & 3 & 4 & \RR \\ \hline
\text{behind }1 & m_{ij}  & 0 \text{ or } \langle m_{ij},v\rangle & \langle m_{ij},v_1\rangle & 0 & 0 & 0 &  \langle m_{ij},v_1\rangle     \\
\text{behind }2 & m_{ij}  & 0 \text{ or } \langle m_{ij},v\rangle &  0 & \langle m_{ij},v_2\rangle & 0 & 0 &    \langle m_{ij},v_2\rangle    \\
\text{behind }3 & m_{ij}  & 0 \text{ or } \langle m_{ij},v\rangle &  0 & 0 & \langle m_{ij},v_3\rangle & 0 &    \langle m_{ij},v_3\rangle     \\
\text{behind }4 & m_{ij}  & 0 \text{ or } \langle m_{ij},v\rangle &  0 & 0 & 0 & \langle m_{ij},v_4\rangle &    \langle m_{ij},v_4\rangle \\  
\end{array}.$$
Now, we see that the last column is the sum of the columns indexed $1,2,3,4$. Thus, the sum of the determinants is $0$. If  some of the edges $1,2,3,4$ was unbounded, the columns with the corresponding indices would not appear, but for the linear forms $m_{ij}$ that would be evaluated on the corresponding unbounded end, one would already have $\langle m_{ij},v_i\rangle=0$ and the result is unchanged. Finally, one has
$$\mathrm{det} A_{12//34}+\mathrm{det} A_{13//24}+\mathrm{det} A_{14//23}=0.$$

\item We now use the previous statement to prove the invariance of the count. We denote by $\mathcal{L}$ the tuple $(\mathcal{L}_e)$. Let $\mathcal{L}(t)$ be a generic path between two generic configurations $\mathcal{L}(0)$ and $\mathcal{L}(1)$, \textit{i.e.} a path in $\prod_e Q_e^\RR$. Outside a finite set of values of $t$, the configuration $\mathcal{L}(t)$ is generic and $N_\Delta(\mathcal{L}(t))$ is given by a sum over some combinatorial types with non-zero multiplicity. More precisely, on each combinatorial type with non-zero multiplicity, the composed evaluation map $\mathrm{ev}^{(L_e)}$ is a linear map associated with a matrix $A$. If the coordinates on the combinatorial type are denoted by $(V,l)$, the equation can be formally solved: $(V,l)=A^{-1}\mathcal{L}$, and this provide a true solution of $\mathrm{ev}^{(L_e)}(\Gamma,h)=\mathcal{L}$ if the coordinates of $l$ are non-negative.\\

As the multiplicity only depends on the combinatorial type, the value of $N_\Delta(\mathcal{L}(t))$ is locally constant at generic $\mathcal{L}(t)$, thus, we only need to show invariance around the special values $t$ where $\mathcal{L}(t)$ is not generic. Let $t^*$ be such a value. At least one of the curves of $(\mathrm{ev}^{(L_e)})^{-1}(t^*)$ has a quadrivalent vertex, and it deforms into the adjacent combinatorial types when $t$ moves slightly around $t^*$. Let $A_{12//34}$, $A_{13//24}$ and $A_{14//23}$ be the matrices of $\mathrm{ev}^{(L_e)}$ on the three adjacent combinatorial types. On each combinatorial type, we can solve uniquely $A_\star(V,l)=\mathcal{L}(t)$: $(V,l)=A_\star^{-1}\lambda(t)$, where $\star$ is one of the three adjacent combinatorial types, and we get a true solution if all the coordinates of $l$ are non-negative. This is the case for all the edges except the edge that appears with the splitting of the quadrivalent vertex. Using Cramer's rule to solve $A_\star(V,l)=\mathcal{L}(t)$, the length of the new edge is equal to $\frac{\mathrm{det}\tilde{A}_\star}{\mathrm{det} A_\star}$, where $\tilde{A}_\star$ is the matrix $A_\star$ with the last column (the one that corresponds to the length of the new edge) being replaced with $\mathcal{L}\in\prod_i Q_e^\RR$. As the matrices $A_{12//34}$, $A_{13//24}$ and $A_{14//23}$ only differ in their last column, the numerators are all equal, and the sign of the length of the new edge is thus determined by the sign of $\mathrm{det} A_\star$. Finally, the sign of the determinant determines which combinatorial type provide a true solution, and the local invariance follows from the relation
$$\mathrm{det} A_{12//34}+\mathrm{det} A_{13//24}+\mathrm{det} A_{14//23}=0.$$
\end{itemize}

\end{proof}

\begin{rem}
The invariance also results from general results of tropical intersection theory \cite{allermann2010first}. Had we worked with more general tropical cycles $\Xi_e$, the proof with intersection theory would also work, but in that case, we would have more walls to study. These walls correspond to the cases where the parametrized tropical curves do not intersect the cycles $\Xi_e$ in their top-dimensional faces. The invariance would then result from the balancing condition for the cycles $\Xi_e$.
\end{rem}

\bibliographystyle{plain}
\bibliography{biblio}

{\ncsc Universit\'e de Neuch\^atel\\[-15.5pt] 

Avenue \'Emile Argan 11,
2000 Neuch\^atel,
Suisse} \\[-15.5pt]

\medskip

{\it E-mail address}: {\ntt thomas.blomme@imj-prg.fr}

\end{document}